\documentclass{amsart}
\usepackage{mathrsfs}
\usepackage{amssymb,latexsym,amsmath,amsthm,enumitem,mathrsfs}
\usepackage[margin=1in]{geometry}
\usepackage{color}
\usepackage{stmaryrd}
\usepackage{mathrsfs}
\usepackage{lineno}
\usepackage{bbm}   
\usepackage{scalerel,stackengine}
\usepackage{ esint }
\usepackage{graphicx}
\usepackage{amssymb}
\newtheorem{theorem}{Theorem}
\newtheorem{lemma}{Lemma}

\newcommand{\R}{\mathbb R}

\newcommand{\Z}{\mathbb Z}
\newcommand{\C}{\mathbb C}
\newcommand{\N}{\mathbb N}
\newcommand{\T}{\mathbb T}

\newcommand{\w}{\omega}

\newcommand{\e}{\varepsilon}
\newcommand{\g}{\gamma}
\newcommand{\p}{\varphi}
\newcommand{\s}{\psi}
\newcommand{\z}{\zeta}

\renewcommand{\a}{\alpha}
\renewcommand{\b}{\beta}
\renewcommand{\d}{\delta}

\renewcommand{\P}{\mathbb{P}}

\newcommand{\mc}{\mathcal}
\newcommand{\mb}{\mathbb}

\def\set4{\mathcal I}
\def\tup14{(1,2,3,4)}

\newcommand\vwidehat[1]{\arraycolsep=0pt\relax%
\begin{array}{c}
\stretchto{
  \scaleto{
    \scalerel*[\widthof{\ensuremath{#1}}]{\kern-.5pt\bigwedge\kern-.5pt}
    {\rule[-\textheight/2]{1ex}{\textheight}} 
  }{\textheight} %
}{0.5ex}\\           
#1\\                 
\rule{-1ex}{0ex}
\end{array}
}
\newtheorem*{comm*}{Comment}
\newtheorem*{rmk}{Remark}
\newtheorem{definition}{Definition}
\newtheorem*{lemma*}{Lemma}
\newtheorem*{theorem*}{Theorem}
\newtheorem{corollary}{Corollary}
\newtheorem{thm}{Theorem}

\newtheorem{conj}[theorem]{Conjecture}
\newtheorem{cor}[thm]{Corollary}
\newtheorem{proposition}{Proposition}

\usepackage{ bbold }
\usepackage{bbm}   
\usepackage{scalerel,stackengine}
\stackMath
\newcommand\widecheck[1]{%
\savestack{\tmpbox}{\stretchto{%
  \scaleto{%
    \scalerel*[\widthof{\ensuremath{#1}}]{\kern-.6pt\bigwedge\kern-.6pt}%
    {\rule[-\textheight/2]{1ex}{\textheight}}
  }{\textheight}%
}{0.5ex}}%
\stackon[1pt]{#1}{\scalebox{-1}{\tmpbox}}%
}
\usepackage{ esint }

\usepackage{enumitem}

\newcommand{\supp}{\mathrm{supp}}

\newtheorem{prop}[theorem]{Proposition}

\begin{document}

\author{Larry Guth}
\address{Department of Mathematics\\
Massachusetts Institute of Technology\\
Cambridge, MA 02142-4307, USA}
\email{lguth@math.mit.edu}

\author{Dominique Maldague}
\address{Department of Mathematics\\
Massachusetts Institute of Technology\\
Cambridge, MA 02142-4307, USA}
\email{dmal@mit.edu}


\date{}

\title{Small cap decoupling for the moment curve in $\R^3$ }
\maketitle

\begin{abstract}
    We prove sharp small cap decoupling estimates for the moment curve in $\R^3$. Our formulation of the small caps is motivated by a conjecture about $L^p$ estimates for exponential sums from \cite{smallcap}. 
\end{abstract}

\section{Introduction} 

We use high/low frequency methods to prove small cap decoupling inequalities for the moment curve $\mc{M}^3=\{(t,t^2,t^3):\,t\in[0,1]\}$ in $\R^3$. We begin by describing the problem and our results in terms of exponential sums. The motivation for this paper is to prove Conjecture 2.5 with $n=3$ from \cite{smallcap}, which we state now. Use the standard notation $e(t)=e^{2\pi i t}$.

\begin{conj}\label{conj} For each $N\ge 1$, $0\le \sigma\le 2$, and $s\ge 1$, 
\[  \int_{[0,1]^2\times[0,\frac{1}{N^{\sigma}}]}|\sum_{k=1}^Ne(kx_1+k^2x_2+k^3x_3)|^{2s}dx\le C_\e N^\e  \Big[N^{s-\sigma}+N^{2s-6}\Big] . \]
\end{conj}
The $s=1$ and $s=\infty$ versions of this conjecture are easily verified using $L^2$-orthogonality and the triangle inequality, respectively. When $\sigma=0$, this is Viongradov's mean value theorem, solved in three dimensions by Wooley \cite{woocubic} and using decoupling for the moment curve by Bourgain, Demeter, and Guth \cite{bdg}. The case of $\sigma=2$ was proven by Bombieri and Iwaniec \cite{bombiwa} and by Bourgain \cite{smallcapb1} using a different argument. In \cite{smallcap}, they prove a slightly more general statement which implies Conjecture \ref{conj} in the range $0\le \sigma\le\frac{3}{2}$. We prove the following general exponential sum estimate which implies Conjecture \ref{conj} for the full range of $\sigma$. 
\begin{theorem}\label{mainexp}
For each $N\ge 1$, $0\le \sigma\le 2$, interval $H$ of length $\frac{1}{N^\sigma}$, and $s\ge 1$, 
\[ \int_{[0,1]^2\times H}|\sum_{k=1}^N a_k e(kx_1+k^2x_2+k^3x_3)|^{2s}dx\le  C_\e N^\e\big[ N^{s-\sigma}+N^{2s-6}\big] \] 
for any $a_k\in\C$ satisfying $|a_k|\lesssim 1$. 
\end{theorem}
The terms in the upper bound come from two examples. The upper bound $N^{s-\sigma}$ follows from taking random $a_\xi\in\{\pm1\}$, by Khintchine's inequality. The upper bound $N^{2s-6}$ follows from the example $a_\xi=1$ and noting that the integrand is $\gtrsim N^{2s}$ on roughly the box $[0,\frac{1}{N}]\times[0,\frac{1}{N^2}]\times[0,\frac{1}{N^3}]$. Theorem \ref{mainexp} is an estimate for the moments of exponential sums over subsets smaller than the full domain of periodicity (i.e. $N^3$ in the $x_3$-variable). Bourgain investigated examples of this type of inequality in \cite{smallcapb1,smallcapb2}. 

Theorem \ref{mainexp} is a corollary of a small cap decoupling problem for $\mc{M}^3$ which we now describe. For $R\ge 1$, and small cap parameter $\b\in[\frac{1}{3},1]$, consider the anisotropic small cap neighborhood 
\[\mc{M}^3(R^\b,R)=\{(\xi_1,\xi_2,\xi_3): \xi_1\in[0,1],\,|\xi_2-\xi_1^2|\le R^{-{2\b}},\,|\xi_3-3\xi_1\xi_2+2\xi_1^3|\le R^{-1} \}.  \]
This is the anisotropic neighborhood of $\mc{M}^3$ at scale $R^\b$ (for which canonical decoupling for the moment curve applies) plus a vertical interval of length $R^{-1}$. Next we define small caps $\g$, which form a partition of $\mc{M}^3(R^\b,R)$ and are defined precisely in \textsection\ref{blocks}. Each $\g$ has the form 
\begin{equation}\label{momblocksintro}
    \g=\{(\xi_1,\xi_2,\xi_3): lR^{-\b}\le \xi_1<(l+1)R^{-\b},\,|\xi_2-\xi_1^2|\le R^{-2\b},\,|\xi_3-3\xi_1\xi_2+2\xi_1^3|\le R^{-1} \} 
\end{equation}
for some integer $l$, $0\le l<R^\b$. When $\b=\frac{1}{3}$, then $\g$ coincides with canonical $R^{-\frac{1}{3}}\times R^{-\frac{2}{3}}\times R^{-1}$ moment curve blocks. In the range $\frac{1}{3}\le\b\le\frac{1}{2}$, $\g$ is essentially equivalent to the $R^{-1}$-neighborhood of a canonical $R^{-\b}\times R^{-2\b}\times R^{-3\b}$ moment curve block. In the range $\frac{1}{2}\le\b\le 1$, $\g$ looks like a canonical $R^{-\b}\times R^{-2\b}\times R^{-3\b}$ moment curve block plus a vertical $R^{-1}$-interval. In each case, $\g$ has dimensions $R^{-\b}\times R^{-2\b}\times R^{-1}$. Our definition of small caps using the vertical $R^{-1}$ neighborhood is motivated by Theorem \ref{mainexp}, which we explain further in \textsection\ref{imps}. 

The small cap decoupling theorem we obtain is 
\begin{theorem} \label{main} Let $\frac{1}{3}\le \b\le 1$ and $p\ge 2$. Then 
\[     \|f\|_{L^p(\R^3)}^p\le C_\e R^\e(R^{\b(\frac{p}{2}-1)}+R^{\b(p-4)-1})\sum_\g\|f_\g\|_{L^p(\R^3)}^p \]
for any Schwartz function $f:\R^3\to\C$ with Fourier transform supported in $\mc{M}^3(R^\b,R)$.  
\end{theorem}
The only other result of this form that we are aware of is the work of Jung in \cite{hjung}, which essentially proves the $\b=\frac{1}{2}$ case of Theorem \ref{main}. The proof of Theorem \ref{main} uses the same framework as the high-low argument from \cite{gmw}. We require a crucial new ingredient, which is small cap decoupling for the cone established in \cite{ampdep}. See \textsection\ref{highsec} for some discussion of the role of small cap decoupling for the cone in the proof of Theorem \ref{main}. In the special case that $\b=\frac{1}{3}$, our argument recovers canonical decoupling for the moment curve in $\R^3$ (first established in \cite{bdg}) using a high-low argument. 

The powers of $R$ in the upper bound of Theorem \ref{main} come from considering two natural sharp examples for the ratio $\|f\|_{p}^p/(\sum_\g\|f_\g\|_p^q)^{p/q}$. The first is the square root cancellation example, where $|f_\g|\sim \chi_{B_{R^{\max(2\b,1)}}}$ for all $\g$ and $f=\sum_\g e_\g f_\g$ where $e_\g$ are $\pm1$ signs chosen (using Khintchine's inequality) so that $\|f\|_{p}^p\sim R^{\b p/2}R^{3\max(2\b,1)}$ and
\[ \|f\|_p^p/(\sum_\g\|f_\g\|_p^p)\gtrsim (R^{\b p/2}R^{3\max(2\b,1)})/(R^{\b }R^{3\max(2\b,1)})\sim R^{\b(\frac{p}{2}-1)}. \]
The second example is the constructive interference example. Let $f_\g=R^{\b+2\b+1}\widecheck{\eta}_\g$ where $\eta_\g$ is a smooth bump function approximating $\chi_\g$. Since $|f|=|\sum_\g f_\g|$ is approximately constant on unit balls and $|f(0)|\sim R^\b$, we have
\[\|f\|_p^p/(\sum_\g\|f_\g\|_p^p)\gtrsim (R^{\b p})/(R^{\b }R^{\b+2\b+1})\sim R^{\b (p-4)-1}. \]

We remark that the arguments in this paper could also be used to analyze the natural problem of small cap decoupling problem with $R^{-1}$ neighborhoods of canonical blocks. That set-up is slightly more complicated than the one considered here since for some choice of parameters $\b$ and $p$, the block example $f=\widecheck{\eta}_\theta$, with $\theta$ a canonical $R^{-\frac{1}{3}}\times R^{-\frac{2}{3}}\times R^{-1}$ block, may also dominate. This leads to extra cases and a more complicated proof that we do not present here. 

An immediate corollary of Theorem \ref{main} is the following general exponential sum estimate. 
\begin{cor}\label{maincor} For each $\frac{1}{3}\le \b\le 1$, $2\le p\le 6+\frac{2}{\b}$, and $r\ge R^{\max(2\b,1)}$, 
\[ |Q_{r}|^{-1}\int_{Q_{r}}|\sum_{\xi\in\Xi}a_\xi e(x\cdot(\xi,\xi^2,\xi^3))|^pdx \lesssim_\e R^{\b\frac{p}{2}+\e}  \]
for any $r$-cube $Q_r$ and any collection $\Xi\subset[0,1]$ with $|\Xi|\sim R^\b$ consisting of $\sim R^{-\b}$-separated points and  $a_\xi\in\C$
with $|a_\xi|\lesssim 1$. 
\end{cor}
Note that the corresponding corollary of canonical decoupling $\mc{M}^3$ only holds in the range $r\ge R^{3\b}$. 

For $a,b>0$, the notation $a\lesssim b$ means that $a\le Cb$ where $C>0$ is a universal constant whose definition varies from line to line, but which only depends on fixed parameters of the problem. Also, $a\sim b$ means $C^{-1}b\le a\le Cb$ for a universal constant $C$, and $a\lesssim_\e b$ means that the implicit constant depends on $\e>0$. 

The paper is organized as follows. We explain the implications of Theorem \ref{main} in \textsection\ref{imps} and give some intuition for the proof of Theorem \ref{main} in \textsection\ref{highsec}. Then in \textsection\ref{tools}, we develop multi-scale high/low frequency tools and lemmas. Some of these tools are very similar to those developed in \cite{gmw}, but the high-frequency analysis uses the geometry of the moment curve and relies on small cap decoupling estimates for the cone recently established in \cite{ampdep}. We use these tools in \textsection\ref{alphapc} to prove a weak (superlevel set) version of Theorem \ref{main} for the critical exponent $p_c=6+\frac{2}{\b}$. Then in \textsection\ref{M3pigeon}, we perform a sequence of pigeonholing steps analogous to those in Section 5 of \cite{gmw} to show that Theorem \ref{main} follows from the superlevel set version with the critical exponent. 

\noindent\textbf{Acknowledgement:} Thank you to Hong Wang for helpful conversations related to this work.  In particular, she suggested to us the idea of using decoupling in place of orthogonality within the high/low method.

LG is supported by a Simons Investigator grant. DM is supported by the National Science Foundation under Award No. 2103249.

\subsection{Implications of Theorem \ref{main}\label{imps}}

\begin{proof}[Corollary \ref{maincor} follows from Theorem \ref{main}]
Let $\phi_{Q_r}$ be a nonnegative Schwartz function satisfying $\phi_{Q_r}\gtrsim 1$ on $Q_r$, $\supp \,\widehat{\phi}_{Q_r}\subset B_{r^{-1}}$, and $\int|\phi_{Q_r}|^p\sim_p |Q_r|$. Then the function 
\[ f(x)=\sum_{\xi\in\Xi}a_\xi e(x\cdot(\xi,\xi^2,\xi^3))\phi_{Q_r}(x) \]
satisfies the hypotheses of Theorem \ref{main}. Using the triangle inequality, we may split the indexing set $\Xi$ into $O(1)$ many subsets $\Xi'$ so that each $\xi\in\Xi'$ is identified with a unique small cap $\g$ which completely contains the $r^{-1}$-neighborhood of $(\xi,\xi^2,\xi^3)$. This is possible because $r\ge R^{\max(2\b,1)}$, so a ball of radius $r^{-1}$ can be completely contained in an $R^{-\b}\times R^{-2\b}\times R^{-1}$ small cap $\g$, whose geometry is described in detail in \textsection\ref{blocks}. Applying Theorem \ref{main} in the range $2\le p\le 6+\frac{2}{\b}$ gives 
\begin{align*}
    \int_{Q_r}|f|^{p}\lesssim_\e R^{\b(\frac{p}{2}-1)+\e}\sum_{\xi\in\Xi}\|a_\xi e(\cdot(\xi,\xi^2,\xi^3))\phi_{Q_R}\|_p^p\sim R^{\b\frac{p}{2}+\e} |Q_r|.
\end{align*}
\end{proof}
\[\]
\begin{proof}[Theorem \ref{mainexp} follows from Theorem \ref{main}]
Begin with the integral on the left hand side of Theorem \ref{mainexp}. Perform the change of variables $(x_1,x_2,x_3)=(\frac{y_1}{N},\frac{y_2}{N^2},\frac{y_3}{N^3})$:
\[ \int_{[0,1]^2\times H}|\sum_{k=1}^N a_k e(x\cdot (k,k^2,k^3))|^{2s}dx=N^{-6} \int_{[0,N]\times[0,N^2]\times N^3H}|\sum_{k=1}^Na_k e(y\cdot (\frac{k}{N},\frac{k^2}{N^2},\frac{k^3}{N^3}))|^{2s}dy.  \]
Using the periodicity of the exponential sum in the first two variables, 
\[\int_{[0,N]\times[0,N^2]\times N^3H}|\sum_{k=1}^Na_k e(y\cdot (\frac{k}{N},\frac{k^2}{N^2},\frac{k^3}{N^3}))|^{2s}dy = N^{-3} \int_{[0,N^3]^2\times N^3H}|\sum_{k=1}^Na_k e(y\cdot (\frac{k}{N},\frac{k^2}{N^2},\frac{k^3}{N^3}))|^{2s}dy . \]
Let $\phi_H$ be a bump function which satisfies $\phi_H\gtrsim 1$ on $[0,N^3]^2\times N^3H$, $\supp\,\widehat{\phi}_H\subset[0,N^{-3}]^2\times[0,N^{\sigma-3}]$, and $\int|\phi_H|^p\sim_p N^{9-\sigma}$. Then \[ \int_{[0,N^3]^2\times N^3H}|\sum_{k=1}^Na_k e(y\cdot (\frac{k}{N},\frac{k^2}{N^2},\frac{k^3}{N^3}))|^{2s}dy\lesssim \int_{\R^3}|\sum_{k=1}^Na_k e(y\cdot (\frac{k}{N},\frac{k^2}{N^2},\frac{k^3}{N^3}))\phi_H(y)|^{2s}dy. \]
Then apply Theorem \ref{main} with $p=2s$, $R=N^{3-\sigma}$, and $\b$ defined by $R^\b=N$, which means that $\b=\frac{1}{3-\sigma}\in[\frac{1}{3},1]$ (since $\sigma\in[0,2]$), giving
\[ \int_{\R^3}|\sum_{k=1}^Na_k e(y\cdot (\frac{k}{N},\frac{k^2}{N^2},\frac{k^3}{N^3}))\phi_H(y)|^{2s}dy\lesssim_\e R^\e[R^{\b(s-1)}+R^{\b(2s-4)-1}]\sum_{k=1}^N|a_k|^{2s}\|\phi_H\|_{2s}^{2s} . \]
Incorporate the extra factors from the substitution and the periodicity steps, and use the assumption $|a_k|\lesssim 1$ and the property $\|\phi_H\|_{2s}^{2s}\sim_s N^{9-\sigma}$ to get the bound
\[ \int_{[0,1]^2\times H}|\sum_{k=1}^N a_k e(x\cdot (k,k^2,k^3))|^{2s}dx\lesssim_\e N^{-9}R^\e[R^{\b(s-1)}+R^{\b(2s-4)-1}]NN^{9-\sigma}.\]
Finally, using the relationship between $R$, $N$, $\b$, and $\sigma$, the upper bound simplifies to 
\[ N^\e[N^{(s-1)}+N^{(2s-4)-(3-\sigma)}]N^{1-\sigma}= N^\e[N^{s-\sigma}+N^{2s-6}] ,\]
as desired. 
\end{proof}

\subsection{Some intuition behind the proof of Theorem \ref{main}\label{highsec}}

Here we describe one of the cases from the proof of Theorem \ref{main} which illustrates the role of small cap decoupling for the cone. 
After a series of standard reductions which are also used in \cite{gmw}, to prove Theorem \ref{main}, it suffices to show that
\begin{equation}\label{simpuppbd} \a^{6+\frac{2}{\b}}|\{x\in B_{R^{\max(2\b,1)}}:\a\le|f(x)|\}|\lesssim_\e R^\e R^{2\b+1}\sum_\g\|f_\g\|_2^2 \end{equation}
where $\a>0$, $B_{R^{\max(2\b,1)}}$ is a ball of radius $R^{\max(2\b,1)}$, and we have the extra assumption that $\|f_\g\|_\infty\lesssim 1$ for all $\g$. The spatial localization to a ball of radius $R^{\max(2\b,1)}$ is natural since this is the smallest size of ball that contains an $R^\b\times R^{2\b}\times R$ wave packet dual to each $\g^*$. Consider the special case of maximal $\a$, so $\a\sim \#\g\sim R^\b$, and call $\{x\in B_{R^{\max(2\b,1)}}:R^\b\sim |f(x)|\}$ the high set $H$. Using a local trilinear restriction estimate for the moment curve, recorded below in Proposition \ref{trirestprop}, we show roughly that  
\[ (R^\b)^6|H|\lesssim \int_{\mc{N}_{R^\b}(H)}|\sum_\g|f_\g|^2(x)|^3dx.  \]
Suppose that on most of $\mc{N}_{R^\b}(H)$, $\sum_\g|f_\g|^2(x)\lesssim |\sum_\g|f_\g|^2*\widecheck{\eta}_{>\frac{1}{2}R^{-\b}}(x)|$ where $\eta_{>\frac{1}{2}R^{-\b}}$ is a smooth approximation of the characteristic function of the set $\frac{1}{2}R^{-\b}\le |\xi|\le 2R^{-\b}$. Each $\widehat{|f_\g|^2}$ is supported in $\g-\g$. Writing $\g(t)=(t,t^2,t^3)$ and using the definition \eqref{momblocksintro}, the support of each $\widehat{|f_\g|^2}\eta_{>\frac{1}{2}R^{-\b}}$ is approximately contained in 
\[    \{A\g'(lR^{-\b})+B\g''(lR^{-\b})+C\g'''(lR^{-\b}):\frac{1}{2}R^{-\b}\le A\le R^{-\b},|B|\le R^{-2\b},|C|\le R^{-1}\}. \]
In \textsection\ref{blocks}, we show that (1) these sets are disjoint for distinct $l\in\{1,\ldots, R^\b\}$, and (2) each of the above sets is contained in the $R^{-\b}$-dilation of a conical small cap. Note that this is not exactly true when $\b=1$, which is why we use use cylinders instead of balls to cut out the low set in the actual argument. Ignoring this technicality, this means that we may apply a small cap decoupling theorem for the cone to bound the integral 
\[ \int_{\mc{N}_{R^\b}(H)}|\sum_\g|f_\g|^2*\widecheck{\eta}_{>\frac{1}{2}R^{-\b}}|^3. \]
Finally, the functions $\sum_\g|f_\g|^2$ and $|\sum_\g|f_\g|^2*\widecheck{\eta}_{>\frac{1}{2}R^{-\b}}|$ are roughly constant on $R^\b$ balls, which implies that for any $p\ge 0$, we have
\[  (R^\b)^6|H|\lesssim \frac{1}{R^{\b p}}\int_{\mc{N}_{R^\b}(H)}|\sum_\g|f_\g|^2*\widecheck{\eta}_{>\frac{1}{2}R^{-\b}}(x)|^{3+p}dx  . \]
This is an important observation since we have more factors of $R^\b$ in the denominator on the right hand side and we may choose $p$ so that $3+p$ is the critical exponent for the scale of conical small caps that we have, thus using the full strength of the small cap decoupling theorem for the cone. Our argument shows that each of these steps can be sharp, which leads to the upper bound \eqref{simpuppbd}.

\section{Tools for the high/low approach to $\mc{M}^3$ \label{tools}}

We perform a high/low frequency analysis of square functions at various scales, incorporating the pruning process for wave packets analogous to \cite{gmw}. We develop language to discuss canonical caps and small caps of various scales, associated wave packets, and averaged versions of functions which satisfy useful locally constant properties. Then we write a series of key lemmas to analyze the high/low frequency portions of averaged, pruned square functions at various scales.

Begin by fixing some notation. Fix a ball $B_{R^{\max(2\b,1)}}$ of radius $R^{\max(2\b,1)}$. The parameter $\a>0$ describes the superlevel set 
\[ U_\a=\{x\in B_{R^{\max(2\b,1)}}:|f(x)|\ge \a\}.\]

Fix $\b\in[\frac{1}{2},1]$ and $R\ge 2$. Let $\e>0$ be given and consider scales $R_k\in 8^\N$ closest to $R^{k\e}$, for $R^{-1/3}\le R_k^{-1/3}\le 1$, and scales $r_k\in 2^\N$ closest to $R^{\frac{1}{3}+k\e}$, for $R^{-\b}\le r_k^{-1}\le R^{-1/3}$. Let $N$ distinguish the index so that $R_N$ is closest to $R$. Since $R$ and $R_N$ differ at most by a factor of $R^\e$, we will ignore the distinction between $R_N$ and $R$ in the rest of the argument. Similarly, assume that $r_M=R^\b$ for some index $M\in\N$. The relationship between the parameters is
\[ 1=R_0\le R_k^{\frac{1}{3}}\le R_{k+1}^{\frac{1}{3}}\le R_N^{\frac{1}{3}}=r_0\le r_m\le r_{m+1}\le r_M=R^\b. \]

Next we fix notation for moment curve blocks and small caps of various sizes. For the explicit definitions, see \textsection\ref{blocks} below. 
\begin{enumerate}
    \item $\{\g\}$ are small caps associated to $R^\b$ and $R$, meaning $\sim R^{-\b}\times R^{-2\b}\times R^{-3\b}$ moment curve blocks plus the set $\{(0,0,z):|z|\le R^{-1}\}$. 
    \item $\{\g_k\}$ are small caps associated to $r_k$ and $R$ (so $\sim r_k^{-1}\times r_k^{-2}\times r_k^{-3}$ moment curve blocks plus $\{(0,0,z):|z|\le R^{-1}\}$).
    \item $\{\theta\}$ are canonical $\sim R^{-\frac{1}{3}}\times R^{-\frac{2}{3}}\times R^{-1}$ moment curve blocks. 
    \item $\{\tau_k\}$ are canonical $R_k^{-\frac{1}{3}}\times R_k^{-\frac{2}{3}}\times R_k^{-1}$ moment curve blocks.  
\end{enumerate}
The specific definitions of $\g,\g_k,\theta,\tau_k$ in \textsection\ref{blocks} provide the additional property that if $\g_k\cap\g_{k+m}\not=\emptyset$, then $\g_{k+m}\subset\g_k$ (and similarly for the $\tau_k$).

We assume throughout this section (actually until \textsection\ref{M3pigeon}) that the $f_\g$ satisfy the extra condition that
\begin{equation}\label{unihyp} 
\frac{1}{2}\le \|f_\g\|_{L^\infty(\R^3)}\le 2\qquad\text{or}\qquad \|f_\g\|_{L^\infty(\R^3)}=0. \end{equation}

\subsection{A pruning step \label{prusec}}

Here we define wave packets for blocks $\g_k,\tau_k$, and prune the wave packets associated to $f_{\g_k},f_{\tau_k}$ according to their amplitudes. 

For each $\g_k$, fix a dual block $\g_k^*$ with dimensions $r_k^{-1}\times r_k^{-2}\times R$ which is comparable to the convex set
\[ \{x\in\R^3:|x\cdot\xi|\le 1\quad\forall\xi\in\g_k-\g_k\}. \]
For each $\tau_k$, fix a dual block $\tau_k^*$ of dimensions $R_k^{1/3}\times R_k^{2/3}\times R_k$ which is comparable to the convex set
\[ \{x\in\R^3:|x\cdot\xi|\le 1\quad\forall\xi\in\tau_k-\tau_k\}. \]
The main difference between dual small caps $\g_k^*$ and dual canonical caps $\tau_k^*$ is that for each $k$,  $\g_k^*=\tilde{\g_k}^*$ if $\g_k,\tilde{\g_k}\subset\theta$, whereas the $\tau_k^*$ are all distinct. 

We will describe wave packet decompositions for small caps $\{\g_k\}$ and for canonical caps $\{\tau_k\}$ in parallel. 
Let $\T_{\g_k},\T_{\tau_k}$ be the collection of tubes $T_{\g_k},T_{\tau_k}$ which are dual to $\g_k,\tau_k$, contain $\g_k^*,\tau_k^*$, and which tile $\R^3$, respectively. Next, define associated partitions of unity $\s_{T_{\g_k}},\s_{T_{\tau_k}}$. 
Let $\p(\xi)$ be a bump function supported in $[-\frac{1}{4},\frac{1}{4}]^3$. For each $m\in\Z^3$, let 
\[ \s_m(x)=c\int_{[-\frac{1}{2},\frac{1}{2}]^3}|\widecheck{\p}|^2(x-y-m)dy, \]
where $c$ is chosen so that $\sum_{m\in\Z^3}\s_m(x)=c\int_{\R^3}|\widecheck{\p}|^2=1$. Since $|\widecheck{\p}|$ is a rapidly decaying function, for any $n\in\N$, there exists $C_n>0$ such that
\[ \s_m(x)\le c\int_{[-\frac{1}{2},\frac{1}{2}]^3}\frac{C_n}{(1+|x-y-m|^2)^n}dy \le \frac{\tilde{C}_n}{(1+|x-m|^2)^n}. \]
Define the partitions of unity $\s_{T_{\g_k}},\s_{T_{\tau_k}}$ associated to $\g_k,{\tau_k}$ to be $\s_{T_{\g_k}}=\s_m\circ A_{\g_k}$ $\s_{T_{\tau_k}}(x)=\s_m\circ A_{\tau_k}$, where $A_{\g_k},A_{\tau_k}$ are linear transformations taking $\g_k^*$,$\tau_k^*$ to $[-\frac{1}{2},\frac{1}{2}]^3$ and $A_{\g_k}(T_{\g_k})=m+[-\frac{1}{2},\frac{1}{2}]^3$, $A_{\tau_k}(T_{\tau_k})=m+[-\frac{1}{2},\frac{1}{2}]^3$. The important properties of $\s_{T_{\g_k}},\s_{T_{\tau_k}}$ are (1) rapid decay off of $T_{\g_k},T_{\tau_k}$ and (2) Fourier support contained in $\g_k,\tau_k$ translated to the origin. 

To prove upper bounds for the size of $U_\a$, we will actually bound the sizes of $\sim \e^{-1}$ many subsets which will be denoted $U_\a\cap H$, $U_\a\cap\Lambda_k$, $U_\a\cap\Omega_k$, and $U_\a\cap L$. The pruning process sorts between important and unimportant wave packets on each of these subsets, as described in Lemma \ref{ftofk} below.

In the following definition, $A_\e\gg 1$ is a large enough (determined by Lemma \ref{ftofk}) constant depending on $\e$ which also satisfies $A_\e\ge D_\e$, where $D_\e$ is given by Lemma \ref{low}. We partition the wave packets $\T_{\g_{k}}=\T_{\g_{k}}^{g}\sqcup\T_{\g_{k}}^{b}$ and $\T_{\tau_k}=\T_{\tau_k}^g\sqcup\T_{\tau_k}^b$ into ``good" and ``bad" sets, and define corresponding versions of $f$, as follows. 


\begin{rmk} In the following definitions, let $K\ge 1$  be a large parameter which will be used to define the broad set in Proposition \ref{mainprop}. 
\end{rmk}
\begin{definition}[Pruning with respect to $\g_k$]\label{thetakprune} Let $f_\g^M=f_\g$ and $f_{\g_{M-1}}^{M}=f_{\g_{M-1}}$. For each $1\le k<M$, let 
\begin{align*} \T_{\g_k}^{g}&=\{T_{\g_k}\in\T_{\g_{k}}:\|\s_{T_{\theta_{k}}}f_{\g_{k}}^{k+1}\|_{L^\infty(R^3)}\le K^3A_\e^{M-k+1}\frac{R^\b}{\a}\}, \\
f_{\g_{k}}^{k}=\sum_{T_{\g_k}\in\T_{\g_k}^{g}}&\s_{T_{\g_k}}f^{k+1}_{\g_k}\qquad\text{and}\qquad  f_{\g_{k-1}}^{k}=\sum_{\g_k\subset\g_{k-1}}f_{\g_k}^k .
\end{align*}
\end{definition}
Recall that $\g_0=\theta=\tau_N$. Once the wave packets corresponding to all of the small caps have been pruned, we have $f^1=\sum_{\g_1}f_{\g_1}^1$. 


\begin{definition}[Pruning with respect to $\tau_k$]\label{taukprune} Let $F^{N+1}=f^1$, $F^{N+1}_{\tau_N}=f^1_{\theta}$. 
For each $1\le k\le N$, let 
\begin{align*} \T_{\tau_k}^{g}&=\{T_{\tau_k}\in\T_{\tau_{k}}:\|\s_{T_{\tau_{k}}}F_{\tau_{k}}^{k+1}\|_{L^\infty(R^3)}\le K^3A_\e^{M+N-k+1}\frac{R^\b}{\a}\}, \\
F_{\tau_{k}}^{k}=\sum_{T_{\tau_k}\in\T_{\tau_k}^{g}}&\s_{T_{\tau_k}}F^{k+1}_{\tau_k}\qquad\text{and}\qquad  F_{\tau_{k-1}}^{k}=\sum_{\tau_k\subset\tau_{k-1}}F_{\tau_k}^k .
\end{align*}
\end{definition}
For each $k$, define the $k$th versions of $f$, $F$ to be $f^k=\underset{\g_k}{\sum} f_{\g_k}^k$ and $F^k=\underset{\tau_k}{\sum}F_{\tau_k}^k$.
\begin{lemma}[Properties of $f^k$ and $F^k$] \label{pruneprop}
\begin{enumerate} 
\item\label{item1} $ | f_{\g_{k}}^k (x) | \le |f_{ \g_{k}}^{k+1}(x)|\lesssim \#\g\subset\g_k$ and $| F_{\tau_{k}}^k (x) | \le |F_{ \tau_{k}}^{k+1}(x)|\lesssim \#\g\subset\tau_k.$
\item \label{item2} $\|f_{\g_k}^k\|_{L^\infty(\R^3)}\le K^3A_\e^{M-k+1}R^{3\e}\frac{R^\b}{\a}$ and $\| F_{\tau_k}^k \|_{L^\infty(\R^3)} \le K^3A_\e^{M+N-k+1}R^{3\e}\frac{R^\b}{\a}$.
\item\label{item3} There is some constant $\underline{C}_\e\lesssim \e^{-2}$ so that $\text{supp} \widehat{f_{\g_k}^{k+1}}\subset\text{supp} \widehat{f_{\g_k}^k}\subset \underline{C}_\e \g_k$ and $ \text{supp} \widehat{F_{\tau_k}^{k+1}}\subset\text{supp} \widehat{f_{\tau_k}^k}\subset \underline{C}_\e \tau_k . $
\end{enumerate}
\end{lemma}
\begin{proof} For the first property, recall that $\sum_{T_{\g_k}\in\T_{\g_k}}\s_{T_{\g_k}},\sum_{T_{\tau_k}\s_{T_{\tau_k}} \in \T_{\tau_k}}$ are partitions of unity so we may iterate the inequalities 
\begin{align*}
|F_{\tau_k}^k|\le |F_{\tau_k}^{k+1}|&\le \sum_{\tau_{k+1}\subset\tau_k}|F_{\tau_{k+1}}^{k+1}|\le\cdots\le \sum_{\tau_N\subset\tau_k}|F_{\tau_N}^N|\le  \sum_{\g_1\subset\tau_k}|f_{\g_1}^1|\\
\text{and }\qquad |f_{\g_1}^1|\le&|f_{\g_1}^2|\le \sum_{\g_2\subset\g_1}|f_{\g_2}^2|\le  \cdots \sum_{\g_N\subset\g_1}|f_{\g_N}^N|\le\sum_{\g\subset\g_1}\|f_{\g}\|_{L^\infty(\R^3)}.  
\end{align*}
Then use the assumption that each $\|f_\g\|_{L^\infty(\R^3)}\lesssim 1$. Now consider the $L^\infty$ bound in the second property.  We write
$$ f_{ \g_k}^k(x) = \sum_{\substack{T_{\g_k} \in \T_{\g_k^h},\\ x \in R^\e T_{\g_k}}} \s_{T_{\g_k}} f_{ \g_k}^{k+1} + \sum_{\substack{T_{\g_k} \in \T_{\g_k}^h,\\ x \notin R^\e T_{\g_k}}} \s_{T_{\g_k}} f_{\g_k}^{k+1}. $$

\noindent The first sum has at most $ R^{3\e}$ terms, and each term has norm bounded by $K^3A_\e^{N-k}\frac{R^\b}{\a}$, by the definition of $\T_{\g_k}^h$.  By the first property, we may trivially bound $f_{\tau_k}^{k+1}$ by $\#\g\subset\tau_k \max_\g\|f_\g\|_\infty\lesssim R$. But if $x \notin R^\e T_{\g_k}$, then $\s_{T_{\g_k}}(x) \le R^{-1000}$. Thus 
\begin{align*} 
|\sum_{\substack{T_{\g_k} \in \T_{\g_k}^h,\\  x \notin R^\e T_{\g_k}}} \s_{T_{\g_k}} f^{k+1}_{ \g_k}|&\le \sum_{\substack{T_{\g_k} \in \T_{\g_k}^h,\\ x \notin R^\e T_{\g_k}}} R^{-500}\s_{T_{\g_k}}^{1/2}(x) \|f^{k+1}_{ \g_k}\|_\infty\le R^{-250}\max_\g\|f_\g\|_\infty. 
\end{align*} 
Since $\a\lesssim|f(x)|\lesssim \sum_\g\|f_\g\|_\infty\lesssim R^\b$, we certainly have $R^{-250}\le \frac{R^\b}{\a}$. The argument for $\|F_{\tau_k}^k\|_{L^\infty(\R^3)}$ is analogous. 

The third property depends on the Fourier supports of $\s_{T_{\g_k}}, \s_{T_{\tau_k}}$, which are contained in $\g_k$, $\tau_k$ shifted to the origin. If each $f_{\g_k}^{k+1}$ has Fourier support in $C\g_k$ (that is, a dilated copy of $\g_k$ by a factor of $C$, taken with respect to its centroid), then $\supp\widehat{f_{\g_k}^k}$ is contained in $(1+C)\g_k$. The same type of argument is true for the claims about $F_{\tau_k}^k$ and $F_{\tau_k}^{k+1}$.

\end{proof}

\begin{definition} \label{M3ballweight} Let $\phi:\R^3\to\R$ be a smooth function supported in $[-\frac{1}{4},\frac{1}{4}]^3$. Define
\[ w(x)=|\widecheck{\phi}|^2(x)+\sum_{k=1}^\infty\frac{1}{(1+k^2)^{100}}\Big(|\widecheck{\phi}|^2(t-k). \]
Let $w(t_1,t_2,t_3)=w_0(t_1)w_0(t_2)w_0(t_3)$ and let $Q=[-\frac{1}{2},\frac{1}{2}]^3$ denote the unit cube centered at the origin. For any set $U=T(B)$ where $T$ is an affine transformation $T:\R^3\to\R^3$, define
\[ w_{U}(x)=|U|^{-1}w(T^{-1}(x)). \]
For $\g_k$, $\tau_k$, let $A_{\g_k}$,$A_{\tau_k}$ be affine transformations taking $\g_k^*$, $\tau_k^*$ to $[-\frac{1}{2},\frac{1}{2}]^3$ and define $\w_{\g_k}$,$\w_{\tau_k}$ by
\[  \w_{\g_k}(x)=|\g_k^*|^{-1}w(R_{\g_k}(x))\qquad\text{and}\qquad \w_{\tau_k}(x)=|\tau_k^*|^{-1}w(R_{\tau_k}(x)). \]
Let the capital-W version of weight functions denote the $L^\infty$-normalized (as opposed to $L^1$-normalized) versions, so for example, for any cube $Q_s$ of sidelength $s$, $W_{Q_s}(x)=|Q_s|w_{Q_s}(x)$. If a weight function has subscript which is only a scale, say $s$, then the functions $w_s,W_s$ are weight function localized to the $s$-cube centered at the origin. We will ignore the distinction between an $s$-ball and an $s$-cube. 
\end{definition}

\begin{rmk}
Note the additional property that $\widehat{w}(\xi_1,\xi_2,\xi_3)$ is supported in $[-\frac{1}{2},\frac{1}{2}]^3$, so $w_{s}$ is Fourier supported in an $s^{-1}$-cube at the origin. Similarly, $\w_{\g_k}$ and $\w_{\tau_k}$ are Fourier supported in $\g_k$ and $\tau_k$ translated to the origin, respectively. The same is true for the $W_{B_s},W_{\g_k^*},W_{\tau_k^*}$ weight functions. Finally, note that if $S_1=T_1(Q)$ and $S_2=T_2(Q)$ where $T_i$ are anisotropic dilations  with respect to the standard basis and $S_1\subset S_2$, then $w_{S_1}*w_{S_2}\lesssim w_{S_2}$. 
\end{rmk}

The weights $\w_{\tau_k}$, $\w_{\theta}=\w_{\tau_N}$, and $w_{s}$ are useful when we invoke the locally constant property. By locally constant property, we mean generally that if a function $f$ has Fourier transform supported in a convex set $A$, then for a bump function $\p_A\equiv 1$ on $A$, $f=f*\widecheck{\p_A}$. Since $|\widecheck{\p_A}|$ is an $L^1$-normalized function which is positive on a set dual to $A$, $|f|*|\widecheck{\p_A}|$ is an averaged version of $|f|$ over a dual set $A^*$. We record some of the specific locally constant properties we need in the following lemma.  
\begin{lemma}[Locally constant property]\label{locconst} For each $\g_k,\tau_k$ and $T_{\g_k}\in\T_{\g_k},T_{\tau_k}\in\T_{\tau_k}$ respectively, 
\begin{align*} 
\|f_{\g_k}\|_{L^\infty(T_{\g_k})}^2\lesssim |f_{\g_k}|^2*\w_{\g_k}(x)\qquad\text{for any}\quad x\in T_{\g_k} \\
\text{and}\qquad \|f_{\tau_k}\|_{L^\infty(T_{\tau_k})}^2\lesssim |f_{\tau_k}|^2*\w_{\tau_k}(x)\qquad\text{for any}\quad x\in T_{\tau_k} .\end{align*}
Also, for any $r_k$-ball $B_{r_k}$ or $R_k^{\frac{1}{3}}$-ball $B_{R_k^{\frac{1}{3}}}$, 
\begin{align*} 
\|\sum_{\g_k}|f_{\g_k}|^2\|_{L^\infty(B_{r_k})}\lesssim \sum_{\g_k}|f_{\g_k}|^2*w_{B_{r_k}}(x)\qquad\text{for any}\quad x\in B_{r_k} \\
\text{and}\qquad \|\sum_{\tau_k}|f_{\tau_k}|^2\|_{L^\infty(B_{R_k^{\frac{1}{3}}})}\lesssim |f_{\tau_k}|^2*w_{B_{R_k^{1/3}}}(x)\qquad\text{for any}\quad x\in B_{R_k^{\frac{1}{3}}} .\end{align*}
\end{lemma}
Because the pruned versions of $f$, $f_{\g_k}$, and $f_{\tau_k}$ have similar Fourier supports as the unpruned versions (see Lemma \ref{pruneprop}), the locally constant lemma applies to the pruned versions as well.

\begin{proof}[Proof of Lemma \ref{locconst}] For the first claim, we write the argument for $f_{\tau_k}$ in detail (the argument for the $f_{\g_k}$ is analogous). Let $\rho_{\tau_k}$ be a bump function equal to $1$ on $\tau_k$ and supported in $2\tau_k$. Then using Fourier inversion and H\"{o}lder's inequality, 
\[ |f_{\tau_k}(y)|^2=|f_{\tau_k}*\widecheck{\rho_{\tau_k}}(y)|^2\le\|\widecheck{\rho_{\tau_k}}\|_1 |f_{\tau_k}|^2*|\widecheck{\rho_{\tau_k}}|(y). \]
Since $\rho_{\tau_k}$ may be taken to be an affine transformation of a standard bump function adapted to the unit ball, $\|\widecheck{\rho_{\tau_k}}\|_1$ is a constant. The function $\widecheck{\rho_{\tau_k}}$ decays rapidly off of $\tau_k^*$, so $|\widecheck{\rho_{\tau_k}}|\lesssim w_{{\tau_k}}$.
Since for any $T_{\tau_k}\in\T_{\tau_k}$, $\w_{\tau_k}(y)$ is comparable for all $y\in T_{\tau_k}$, we have
\begin{align*} \sup_{x\in T_{\tau_k}}|f_{\tau_k}|^2*\w_{\tau_k}(x)&\le \int|f_{\tau_k}|^2(y)\sup_{x\in T_{\tau_k}}\w_{\tau_k}(x-y)dy\\
&\sim \int|f_{\tau_k}|^2(y)\w_{\tau_k}(x-y)dy\qquad \text{for all}\quad x\in T_{\tau_k}. 
\end{align*}

For the second part of the lemma, repeat analogous steps as above, except begin with $\rho_{r_k}$ which is identically $1$ on a ball of radius $2r_k^{-1}$ containing $\g_k-\g_k$ (which is the Fourier support of $|f_{\g_k}|^2$). Then 
\[  \sum_{\g_k}|f_{\g_k}(y)|^2=|\sum_{\g_k}|f_{\g_k}|^2*\widecheck{\rho_{r_k}}(y)|\lesssim \sum_{\g_k}|f_{\g_k}|^2*|\widecheck{\rho_{r_k}}|(y). \]
The rest of the argument is analogous to the first part. The argument for $\sum_{\tau_k}|f_{\tau_k}|^2$ is the same.

\end{proof}

For ease of future reference, we record the following standard local and global $L^2$-orthogonality lemma. For $U\subset\R^3$, let $U^*=\{\xi\in\R^3:|\xi\cdot x|\le 1\qquad\forall x\in U-U\}$.  
\begin{lemma}[Local and global $L^2$ orthogonality]\label{L2orth} Let $U=T(Q)$ where $Q$ is the unit ball centered at the origin and $T:\R^3\to\R^3$ is an affine transformation. Let $h:\R^3\to\C$ be a Schwartz function with Fourier transform supported in a disjoint union $X=\sqcup_k X_k$, where $X_k\subset B$ are Lebesgue measurable. If the maximum overlap of the sets $U^*+X_k$ is $L$, then
\[ \int |h_X|^2w_U\lesssim L\sum_{X_k}\int|h_{X_k}|^2w_U, \]
where $h_{X_k}=\int_{X_k}\widehat{h}(\xi)e^{2\pi i x\cdot\xi}d\xi$.
The corresponding global statement is 
\[ \int|h_X|^2=\sum_{X_k}\int|h_{X_k}|^2. \]
\end{lemma}

\begin{proof} The global statement is just Plancherel's theorem.
 For the local statement, we have
\begin{align*}
    \int|h_X|^2w_U&=\int h_X \overline{h_Xw_U}=\int \widehat{h_X}\overline{\widehat{h_X}*\widehat{w_U}}
\end{align*}
by Plancherel's theorem again. Next we used the definition of $\widehat{h_X}$ and $\widehat{h_{X_k}}$ to write
\[\int \widehat{h_X}\overline{\widehat{h_X}*\widehat{w_U}}=\sum_{X_k}\sum_{X_k'}\int\widehat{h_{X_k}}\overline{\widehat{h_{X_k'}}*\widehat{w_U}}.  \]
The function $\widehat{h_{X_k}}$ is supported in $X_k$ and the function $\widehat{h_{X_k'}}*\widehat{w_U}$ is supported in $X_k'+U^*$. Write $X_k'\sim X_k$ to denote the property that $(X_k+U^*)\cap(X_k'+U^*)\not=\emptyset$. By hypothesis, for each $X_k$, there are at most $L$ many $X_k'$ such that $X_k'\sim X_k$. Since $X_k\cap(X_k'+U^*)\subset(X_k+U^*)\cap(X_k'+U^*)$, this leads to the bound
\begin{align*} \sum_{X_k}\sum_{X_k'}\int\widehat{h_{X_k}}\overline{\widehat{h_{X_k'}}*\widehat{w_U}} &= \sum_{X_k}\sum_{X_k'\sim X_k}\int h_{X_k}\overline{h_{X_k'}}w_U\\
    &\le \sum_{X_k}\sum_{X_k'\sim X_k}\int (|h_{X_k}|^2+|h_{X_k'}|^2)w_U \\
    &\le \sum_{X_k}\sum_{X_k'\sim X_k}\int (|h_{X_k}|^2+|h_{X_k'}|^2)w_U\le 2L\sum_{X_k}\int|h_{X_k}|^2w_U. 
\end{align*}

\end{proof}

\begin{definition}[Auxiliary functions] For $i=1,2$, let $\p_i:\R^{i}\to[0,\infty)$ be a radial, smooth bump function satisfying $\p_i(x)=1$ on the unit ball in $\R^{i}$ and supported in the ball of radius $2$. Then for each $\s>0$, let $\rho:\R^3\to[0,\infty)$ be defined by 
\[ \rho_{\le s^{-1}}(\xi_1,\xi_2,\xi_3)=\p_2(s(\xi_1,\xi_2))\p_1(\xi_3). \]
Write ${\mc{C}}_{s^{-1}}$ for the set where $\rho_{\le s^{-1}}=1$. 
\end{definition}
We will sometimes abuse the notation from the previous definition by writing $h*\widecheck{\rho}_{>s^{-1}}=h-h*\widecheck{\rho}_{\le s^{-1}}$.
\vspace{3mm}
\begin{definition} Let $g_M(x)=\sum_{\g}|f_\g|^2*\w_{\g}(x)$. For $1\le k\le M-1$, let 
\[ g_k(x)=\sum_{\g_k}|f_{\g_k}^{k+1}|^2*\w_{\g_k}, \qquad g_k^{\ell}(x)=g_k*\widecheck{\rho}_{\le r_{k+1}^{-1}}, \qquad\text{and}\qquad g_k^h=g_k-g_k^{\ell}. \]
For $1\le k\le N$, let
\[ G_k(x)=\sum_{\tau_k}|F_{\tau_k}^{k+1}|^2*\w_{\tau_k},\qquad G_k^{\ell}(x)=G_k*\widecheck{\rho}_{\le R_{k+1}^{-1/3}},\qquad \text{and}\qquad G_k^h(x)=G_k-G_k^{\ell}.  \]
\end{definition}
\vspace{3mm}

In the following definition, $A_\e\gg 1$ is the same $\e$-dependent constant from the pruning definition of $f^k$ and $F^k$. 
\begin{definition} \label{impsets}Define the high set by 
\[ H=\{x\in B_{R^{\max(2\b,1)}}: A_\e R^\b \le g_{M-1}(x)\}. \]
For each $k=1,\ldots,M-2$, let $H=\Lambda_{M-1}$ and let
\[ \Lambda_k=\{x\in B_{R^{\max(2\b,1)}}\setminus \cup_{l=k+1}^{M-1}\Lambda_{l}: (A_\e)^{(M-k)}R^\b\le g_k(x) \}. \] 
For each $k=1,\ldots,N$, let $\Omega_{N+1}=\cup_{l=1}^{M-1}\Lambda_l$ and let
\[ \Omega_k=\{x\in B_{R^{\max(2\b,1)}}\setminus\cup_{l=k+1}^{N+1}\Omega_{l}:  (A_\e)^{(M+N-k)}R^\b\le G_k(x)\}.  \]
Define the low set to be
\[ L=B_{R^{\max(2\b,1)}}\setminus[(\cup_{l=1}^{N+1}\Omega_N)\cup(\cup_{k=1}^{M-1}\Lambda_k)]. \]
\end{definition}
\vspace{3mm}

\subsection{Lemmas related to the pruning process for wave packets }

\begin{lemma}[Low lemma]\label{low} There is a constant $D=D_\e>0$ depending on $\e$ so that for each $x$, $|g_k^\ell(x)|\le D_\e g_{k+1}(x)$ and $|G_k^\ell(x)|\le D_\e G_{k+1}(x)$. 
\end{lemma}
\begin{proof} Prove the claim in detail for $g_k^\ell$ since the argument for $G_k^\ell$ is analogous. We perform a pointwise version of the argument in the proof of local/global $L^2$-orthogonality (Lemma \ref{L2orth}). For each $\g_k^{k+1}$, using Plancherel's theorem,
\begin{align}
|f_{\g_k}^{k+1}|^2*\widecheck{\rho}_{\le r_{k+1}^{-1}}(x)&= \int_{\R^3}|f_{\g_k}^{k+1}|^2(x-y)\widecheck{\rho}_{\le r_{k+1}^{-1}}(y)dy \nonumber \\
&=  \int_{\R^3}\widehat{f_{\g_k}^{k+1}}*\widehat{\overline{f_{\g_k}^{k+1}}}(\xi)e^{2\pi i x\cdot\xi}\rho_{\le r_{k+1}^{-1}}(\xi)d\xi \nonumber \\
&=  \sum_{\g_{k+1},\g_{k+1}'\subset\g_k}\int_{\R^3}e^{2\pi i x\cdot\xi}\widehat{f_{\g_{k+1}}^{k+1}}*\widehat{\overline{f_{\g_{k+1}'}^{k+1}}}(\xi)\rho_{\le r_{k+1}^{-1}}(\xi)d\xi .\label{dis2}\nonumber
\end{align}
The integrand is supported in $(\underline{C}_\e\g_{k+1}-\underline{C}_\e\g_{k+1}')\cap (2{\mc{C}}_{ r_{k+1}^{-1}})$ where $\underline{C}_\e$ comes from \eqref{item3} of Lemma \ref{pruneprop} and $2{\mc{C}}_{ r_{k+1}^{-1}}$ contains the support of $\rho_{\le r_{k+1}^{-1}}$. The set ${\mc{C}}_{ r_{k+1}^{-1}}$ is contained in a cylinder with a vertical axis, centered at the origin and of radius $2r_{k+1}^{-1}$. The distance between the sets $\underline{C}_\e \g_{k+1}$ and $\underline{C}_\e\g_{k+1}'$ is controlled by the distance of their projections to the $(\xi_1,\xi_2)$-plane. This means that the final integral displayed above vanishes unless $\g_{k+1}$ is within $\sim \underline{C}_\e r_{k+1}^{-1}$ of $\g_{k+1}'$, in which case we write $\g_{k+1}\sim\g_{k+1}'$. Then 
\[\sum_{\g_{k+1},\g_{k+1}'\subset\g_k}\int_{\R^3}e^{2\pi i x\cdot\xi}\widehat{f}_{\g_{k+1}}^{k+1}*\widehat{\overline{f}_{\g_{k+1}'}^{k+1}}(\xi)\rho_{\le r_{k+1}^{-1}}(\xi)d\xi=\sum_{\substack{\g_{k+1},\g_{k+1}'\subset\g_k\\
\g_{k+1}\sim\g_{k+1}'}}\int_{\R^3}e^{2\pi i x\cdot\xi}\widehat{f}_{\g_{k+1}}^{k+1}*\widehat{\overline{f}_{\g_{k+1}'}^{k+1}}(\xi)\rho_{\le r_{k+1}^{-1}}(\xi)d\xi. \]
Use Plancherel's theorem again to return to a convolution in $x$ and conclude that
\begin{align*}
|g_k*\widecheck{\rho}_{\le r_{k+1}^{-1}}(x)|&=\Big|\sum_{\substack{\g_{k+1},\g_{k+1}'\subset\g_k\\
\g_{k+1}\sim\g_{k+1}'}}(f_{\g_{k+1}}^{k+1}\overline{f_{\g_{k+1}'}^{k+1}})*\w_{\tau_k}*\widecheck{\rho}_{\le r_{k+1}^{-1}}(x) \Big|\lesssim \underline{C}_\e\sum_{\g_k} \sum_{\g_{k+1}\subset\g_k}|f_{\g_{k+1}}^{k+1}|^2*\w_{\tau_k}*|\widecheck{\rho}_{\le r_{k+1}^{-1}}|(x)
. 
\end{align*}
By the locally constant property (Lemma \ref{locconst}) and \eqref{item1} of Lemma \ref{pruneprop},
\[ \sum_{\g_k} \sum_{\g_{k+1}\subset\g_k}|f_{\g_{k+1}}^{k+1}|^2*\w_{\tau_k}*|\widecheck{\rho}_{\le r_{k+1}^{-1}}|(x)\lesssim \sum_{\g_k} \sum_{\g_{k+1}\subset\g_k}|f_{\g_{k+1}}^{k+2}|^2*w_{\g_{k+1}}*\w_{\tau_k}*|\widecheck{\rho}_{\le r_{k+1}^{-1}}|(x)\lesssim g_{k+1}(x). \]
It remains to note that
\[ w_{\g_{k+1}}*\w_{\g_k}*|\widecheck{\rho}_{\le r_{k+1}^{-1}}|(x)\lesssim w_{\g_{k+1}}(x) \]
since $\g_k^*$ is comparable to a dilation of $g_{k+1}^*$ and and $\widecheck{\rho}_{\le r_{k+1}^{-1}}$ is an $L^1$-normalized function that is rapidly decaying away from $B_{r_{k+1}}$ (actually, it decays rapidly away from the small set $B^{(2)}_{r_{k+1}}(0)\times B^{(1)}_1(0)$). 

\end{proof}

\begin{corollary}[High-dominance on $\Lambda_k$,$\Omega_k$]\label{highdom} For $R$ large enough depending on $\e$, 
\[ g_k(x)\le 2|g_k^h(x)| \quad\forall x\in\Lambda_k\qquad\text{and}\qquad G_k(x)\le 2|G_k^h(x)|\qquad\forall x\in\Omega_k. \]
\end{corollary}
\begin{proof}
This follows directly from Lemma \ref{low}. Indeed, since $g_k(x)=g_k^{\ell}(x)+g_k^h(x)$, the inequality $g_k(x)>2|g_k^h(x)|$ implies that $g_k(x)<2|g_k^{\ell}(x)|$. Then by Lemma \ref{low}, $|g_k(x)|<2D_\e g_{k+1}(x)$. Since $x\in\Lambda_k$, $g_{k+1}(x)\le A_\e^{M-k-1}R^\b$, or in the case that $k=M-1$, $g_M(x)=\sum_\g |f_\g|^2*\w_\g(x)\lesssim\|\sum_\g|f_\g|^2\|_\infty\lesssim  R^\b$ using the assumption that $\|f_\g\|_\infty\lesssim 1$ for all $\g$. Altogether gives the upper bound
\[ g_k(x)\le 2D_\e A_\e^{M-k-1}R^\b. \]
The contradicts the property that on $\Lambda_k$, $A_\e^{M-k}R^\b\le g_k(x)$, for $A_\e$ sufficiently larger than $D_\e$, which finishes the proof. The argument for $G_k$ on $\Omega_k$ is analogous. 

\end{proof}

\begin{lemma}[Pruning lemma]\label{ftofk} For any $\tau$, 
\begin{align*} 
|\sum_{\g_k\subset\tau}f_{\g_k}-\sum_{\g_k\subset\tau}f_{\g_k}^{k+1}(x)|&\le \frac{\a}{A_\e^{1/2}K^3} \qquad\text{for all $x\in \Lambda_k$},\\
|\sum_{\tau_k\subset\tau}f_{\tau_k}-\sum_{\tau_k\subset\tau}F_{\tau_k}^{k+1}(x)|&\le \frac{\a}{A_\e^{1/2}K^3}\qquad\text{for all $x\in\Omega_k$},\\
\text{and}\qquad |\sum_{\tau_1\subset\tau}f_{\tau_1}-\sum_{\tau_1\subset\tau}F_{\tau_1}^{1}(x)|&\le \frac{\a}{A_\e^{1/2}K^3}\qquad \text{ for all $x\in L$}. \end{align*}
\end{lemma}

\begin{proof} 

Begin by proving the claim about $\Lambda_k$. By the definition of the pruning process, we have 
\begin{equation}\label{prunedecomp} f_{\tau}=f^{M-1}_{\tau}+(f_{\tau}^M-f^{M-1}_{\tau})=\cdots=f^{k+1}_{\tau}(x)+\sum_{m=k+1}^{M-1}(f^{m+1}_{\tau}-f^{m}_{\tau})\end{equation}
where here, the subscript $\tau$ means $f_\tau=\sum_{\g\subset\tau}f_\g$ and $f_{\tau}^m=\sum_{\g_m\subset\tau}f_{\g_m}^m$. We will show that each difference in the sum is much smaller than $\a$.
For each $M-1\ge m\ge k+1$ and $\g_m$, use the notation $\T_{\g_m}^b=\T_{\g_m}\setminus\T_{\g_m}^g$ and write
\begin{align*}
    |f_{\g_m}^m(x)-f_{\g_m}^{m+1}(x)|&=|\sum_{T_{\g_m}\in\T_{\g_m}^{b}}\s_{T_{\g_m}}(x)f_{\g_m}^{m+1}(x)|  = \sum_{T_{\g_m}\in T_{\g_m}^b} |\s_{T_{\g_m}}^{1/2}(x)f_{\g_m}^{m+1}(x)|\s_{T_{\g_m}}^{1/2}(x) \\
     & \le\sum_{T_{\g_m}\in \T_{\g_m}^b}  K^{-3}A_\e^{-(M-m+1)}\frac{\a}{R^\b} \| \s_{T_{\g_m}}f_{{\g_m}}^{m+1} \|_{L^\infty(\R^3)} \| \s_{T_{\g_m}}^{1/2}f_{{\g_m}}^{m+1} \|_{L^\infty(\R^3)}  \s_{T_{\g_m}}^{1/2}(x) \\
     & \lesssim K^{-3}A_\e^{-(M-m+1)}\frac{\a}{R^\b} \sum_{T_{\g_m}\in \T_{\g_m}^b} \| \s_{T_{\g_m}}^{1/2}f_{{\g_m}}^{m+1} \|_{L^\infty(\R^3)}^2 \s_{T_{\g_m}}^{1/2}(x) \\
     & \lesssim K^{-3}A_\e^{-(M-m+1)}\frac{\a}{R^\b}\sum_{T_{\g_m}\in \T_{\g_m}^b}
      \sum_{\tilde{T}_{{\g_m}}} \| \s_{T_{\g_m}}|f_{{\g_m}}^{m+1}|^2 \|_{L^\infty(\tilde{T}_{{\g_m}})} \s_{T_{\g_m}}^{1/2}(x) \\
     & \lesssim K^{-3}A_\e^{-(M-m+1)}\frac{\a}{R^\b} \sum_{T_{\g_m},\tilde{T}_{\g_m}\in \T_{\g_m}} \| \s_{T_{\g_m}}\|_{L^\infty(\tilde{T}_{\g_m})}\||f_{{\g_m}}^{m+1} |^2\|_{{L}^\infty(\tilde{T}_{{\g_m}})} \s_{T_{\g_m}}^{1/2}(x) .
\end{align*}
Let $c_{\tilde{T}_{\g_m}}$ denote the center of $\tilde{T}_{\g_m}$ and note the pointwise inequality
\[ \sum_{{T}_{\g_m}}\|\s_{T_{\g_m}}\|_{L^\infty(\tilde{T}_{\g_m})}\s_{T_{\g_m}}^{1/2}(x)\lesssim |\g_m^*|\w_{\g_m}(x-c_{\tilde{T}_{\g_m}}) ,\]
which means that
\begin{align*}
|f_{\g_m}^m(x)-f_{\g_m}^{m+1}(x)| & \lesssim K^{-3}A_\e^{-(M-m+1)}\frac{\a}{R^\b} |\g_m^*|\sum_{\tilde{T}_{\g_m}\in \T_{\g_m}} \w_{\g_m}(x-c_{\tilde{T}_{\g_m}})\||f_{{\g_m}}^{m+1} |^2\|_{{L}^\infty(\tilde{T}_{{\g_m}})} \\
&\lesssim_\e K^{-3}A_\e^{-(M-m+1)}\frac{\a}{R^\b}|\g_m^*| \sum_{\tilde{T}_{\g_m}\in \T_{\g_m}} \w_{\g_m}(x-c_{\tilde{T}_{\g_m}})|f_{{\g_m}}^{m+1} |^2*\w_{\g_m}(c_{\tilde{T}_{\g_m}})\\
&\lesssim_\e K^{-3}A_\e^{-(M-m+1)}\frac{\a}{R^\b} |f_{{\g_m}}^{m+1} |^2*\w_{\g_m}(x)
\end{align*}
where we used the locally constant property in the second to last inequality. The last inequality is justified by the fact that $\w_{\g_m}(x-c_{\tilde{T}_{\g_m}})\sim \w_{\g_m}(x-y)$ for any $y\in\tilde{T}_{\g_m}$, and we have the pointwise relation $\w_{\g_m}*\w_{\g_m}\lesssim \w_{\g_m}$. The last two inequalities incorporate a dependence on $\underline{C}_\e$ from Lemma \ref{pruneprop} since the locally constant property uses that $\widehat{|f_{\g_m}^{m+1}|^2}$ is supported in the $\underline{C}_\e$-dilation of $\g_m-\g_m$. It is important to note that $\underline{C}_\e$ is a combinatorial factor that does not depend on $A_\e$. 
Then 
\[
    |\sum_{\g_m\subset\tau}f_{\g_m}^m(x)-f_{\g_m}^{m+1}(x)|\lesssim_\e K^{-3}A_\e^{-(M-m+1)}\frac{\a}{R^\b}\sum_{\g_m\subset\tau}|f_{\g_m}^{m+1}|^2*\w_{\g_m}(x)\sim_\e K^{-3}A_\e^{-(M-m+1)}\frac{\a}{R^\b}g_m(x). \]
At this point, choose $A_\e$ large enough so that if $g_m(x)\le A_\e^{M-m}R^\b$, then the above inequality implies that
\[ |\sum_{\g_m\subset\tau}f_{\g_m}^m(x)-f_{\g_m}^{m+1}(x)|\le  \e K^{-3}A_\e^{-1/2}\a  .\]
This finishes the proof since $M+N\lesssim\e^{-1}$, so the number of steps from \eqref{prunedecomp} is controlled. 
The argument for the pruning on $\Omega_k$ and on $L$ is analogous.  
\end{proof}

\subsection{Geometry related to the high frequency parts of square functions \label{blocks}}

We have seen in Corollary \ref{highdom} that on $\Lambda_k$ and $\Omega_k$, $g_k$ and $G_k$ are high-dominated. In this subsection, we describe the geometry of the Fourier supports of $g_k^h$ and $G_k^h$, which will allow us to apply certain decoupling theorems for the cone in \textsection\ref{hisec}. We begin with the precise definitions of canonical blocks and small cap blocks (which we also call ``small caps") of the moment curve. 

\begin{definition}[Canonical moment curve blocks]
For $S\in 2^{\N}$, $S\ge10$, consider the anisotropic neighborhood 
\[\mc{M}^3(S)=\{(\xi_1,\xi_2,\xi_3): \xi_1\in[0,1],\,|\xi_2-\xi_1^2|\le S^{-2},\,|\xi_3-3\xi_1\xi_2+2\xi_1^3|\le S^{-3} \}.  \]
Define canonical moment curve blocks at scale $S$ which partition $\mc{M}^3(S)$ as follows:
\[ \bigsqcup\limits_{l=0}^{S-1}\{(\xi_1,\xi_2,\xi_3): lS^{-1}\le \xi_1<(l+1)S^{-1},\,|\xi_2-\xi_1^2|\le S^{-2},\,|\xi_3-3\xi_1\xi_2+2\xi_1^3|\le S^{-3} \}.  \]
\end{definition}

\begin{definition}[``Small caps" of the moment curve]
Let $R\ge 10$ and let $S\in 2^\N$ satisfy $R^{-1}\le S^{-1}\le R^{-\frac{1}{3}}$. Consider the anisotropic small cap neighborhood
\[ \mc{M}^3(S,R)=\{(\xi_1,\xi_2,\xi_3): \xi_1\in[0,1],\,|\xi_2-\xi_1^2|\le S^{-2},\,|\xi_3-3\xi_1\xi_2+2\xi_1^3|\le R^{-1} \} .\]
Define small caps $\g$ associated to the parameters $S$ and $R$ by  
\begin{equation}\label{momblocks}
    \sqcup\g=\bigsqcup\limits_{l=0}^{S-1}\{(\xi_1,\xi_2,\xi_3): lS^{-1}\le \xi_1<(l+1)S^{-1},\,|\xi_2-\xi_1^2|\le S^{-2},\,|\xi_3-3\xi_1\xi_2+2\xi_1^3|\le R^{-1} \}. 
\end{equation}
\end{definition}
Note that the small caps $\g$ are essentially canonical moment curve blocks at scale $S$ plus a vertical ($\xi_3$-direction) $R^{-1}$-neighborhood.

To analyze $g_k^h$, we need to understand the Fourier support of $\sum_{\g_k}|f_{\g_k}^{k+1}|^2$ outside of a cylinder of radius $r_{k+1}^{-1}$. By \eqref{item3} of Lemma \ref{pruneprop}, the support of $\widehat{|f_{\g_k}^{k+1}|^2}$ is $\underline{C}_\e\g_k-\underline{C}_\e\g_k$. Suppose that $\g_k$ is the $l$th piece, meaning that 
\[ \g_k=\{(\xi_1,\xi_2,\xi_3): lr_k^{-1}\le \xi_1<(l+1)r_k^{-1},\,|\xi_2-\xi_1^2|\le r_k^{-2},\,|\xi_3-3\xi_1\xi_2+2\xi_1^3|\le R^{-1} \}   \]
where $l\in\{0,\ldots,r_k-1\}$. The small cap $\g_k$ is comparable to the set 
\[ \underline{\g_k}=\{\g(lr_k^{-1})+A\g'(lr_k^{-1})+B\g''(lr_k^{-1})+C\g'''(lr_k^{-1}):0\le A\le r_k^{-1},\quad |B|\le r_k^{-2},\quad |C|\le R^{-1}\} \]
in the sense that $\frac{1}{20}{\underline{\g_k}}\subset\g_k\subset 20\underline{\g_k}$ (where the dilations are taken with respect to the centroid of $\g_k$).
Then $\g_k-\g_k$ is contained in 
\[ \{A\g'(lr_k^{-1})+B\g''(lr_k^{-1})+C\g'''(lr_k^{-1}):|A|\lesssim r_k^{-1},\quad |B|\lesssim r_k^{-2},\quad |C|\lesssim R^{-1}\}. \]
Recall that $1-\rho_{\le r_{k+1}^{-1}}$ is supported outside ${\mc{C}}_{r_{k+1}^{-1}}\supseteq\{|(\xi_1,\xi_2)|\le r_{k+1}^{-1}\}$. Intersecting $\underline{C}_\e\g_k-\underline{C}_\e\g_k$ with the support of $1-\rho_{\le r_{k+1}^{-1}}$ forces the relation $A^2+(A2(lr_k^{-1})+2B)^2\ge r_{k+1}^{-2}$. Using the upper bounds $|A|\lesssim \underline{C}_\e r_k^{-1}$ and $|B|\lesssim\underline{C}_\e r_k^{-2}$, it follows that for $R$ large enough depending on $\e$, the support of the high-frequency part of $\widehat{|f_{\g_k}^{k+1}|^2}$ is contained in 
\begin{align}\label{set} \tilde{\g}_k:=\{A\g'(lr_k^{-1})+&B\g''(lr_k^{-1})+C\g'''(lr_k^{-1}):\\ &\qquad\frac{1}{2}r_{k+1}^{-1}\le |A|\lesssim \underline{C}_\e r_k^{-1},
|B|\lesssim\underline{C}_\e r_k^{-2},\quad |C|\lesssim \underline{C}_\e R^{-1}\}. \nonumber\end{align}

Our ``high lemmas" will require geometric properties that are recorded in the following propositions.

\begin{proposition} \label{geo1}The sets $\tilde{\g}_k$, varying over $\g_k$, are $\le C_\e R^\e $-overlapping. 
\end{proposition}
\begin{proof} Suppose that a point corresponding to parameters $A,B,C,l$ and $A',B',C',l'$ respectively is in the intersection of two sets as in \eqref{set}. By analyzing the first coordinate, we must have $A=A'$. By analyzing the second coordinate, we must have
\[ |A2lr_k^{-1}-A2l'r_k^{-1}|\lesssim \underline{C}_\e r_k^{-2}.\]
Therefore, since $A\gtrsim r_{k+1}^{-1}$, $|l-l'|\lesssim \underline{C}_\e R^\e$. 
\end{proof}


Next we describe the geometry of a small cap partition for the cone. Let $\b_1\in[\frac{1}{2},1]$ and $\rho\ge 1$. Let $S\in 2^\N$ a dyadic number closest to $\rho^{\b_1}$. For the (truncated) cone $\Gamma=\{\xi:\xi_1^2+\xi_2^3=\xi_3^2,\quad\frac{1}{2}\le \xi_3\le 1\}$, divide $[0,2\pi)$ into $S$ many intervals $I_S$ of length $2\pi/S$ and define the small cap partition 
\[ \mc{N}_{S^{-1}}(\Gamma)=\underset{I_S}{\sqcup}\mc{N}_{S^{-1}}(\Gamma)\cap\{(\rho \cos\zeta,\rho\sin\zeta,z):\zeta\in I_S)\} \]
corresponding to parameters $\b_1$ and $\b_2=0$, as in Theorem 3 from \cite{ampdep}. After a linear transformation, we will identify the high parts of sets $\g_k-\g_k$ as subsets of conical small caps.

\begin{proposition} \label{geo2} Let $r^{-1}\in[r_{k+1}^{-1},20\underline{C}_\e r_k^{-1}]$ be a dyadic value and write $\{\xi_3\sim r^{-1}\}:=\{(\xi_1,\xi_2,\xi_3)\in\R^3:\frac{r^{-1}}{2}\le \xi_3\le r^{-1}\}$. There is an affine transformation $T:\R^3\to\R^3$ so that the following holds. 
\begin{enumerate}
    \item If $r_k^{-1}\le R^{-\frac{1}{2}}$, then the collection of $\g_k$ may be partitioned into $\lesssim_\e R^{2\e}$ many subsets $\mc{S}_i$ which satisfy the following. For each $\mc{S}_i$, there is a conical small cap partition of $\sim  1\times \underline{C}_\e \frac{r}{R}\times \underline{C}_\e\frac{r}{R}$ blocks so that for each $\g_k\in\mc{S}_i$, $r[T(\tilde{\g}_k)\cap\{\xi_3\sim r^{-1}\}]$ is completely contained in one of the conical small caps. 
    \item If $R^{-\frac{1}{2}}\le r_k^{-1}$ and $(Rr_k^{-1})^{-\b_1}=r_k^{-1}$ for some $\b_1\in[\frac{1}{2},1]$, then the collection of $\g_k$ may be partitioned into $\lesssim_\e R^{2\e}$ many subsets $\mc{S}_i$ which satisfy the following. For each $\mc{S}_i$, there is a conical small cap partition of $\sim  1\times \underline{C}_\e (\frac{r}{R})^{\b_1}\times \underline{C}_\e^{\b_1^{-1}}\frac{r}{R}$ blocks so that each $r[T(\tilde{\g}_k)\cap\{\xi_3\sim r^{-1}\}]$, where $\g_k\subset\mc{S}_i$, is completely contained in one of the conical small caps. 
\end{enumerate}
\end{proposition}

\begin{proof} Let $T:\R^3\to\R^3$ be the affine transformation
\[ T(x,y,z):=(\frac{y}{2},\frac{x-\frac{z}{6}}{\sqrt{2}},\frac{x+\frac{z}{6}}{\sqrt{2}}). \]
The image of the set \eqref{set} under $T$ is
\begin{align*}
T(\tilde{\g}_k)=\{A(lr_k^{-1},\frac{1-\frac{l^2r_k^{-2}}{2}}{\sqrt{2}},\frac{1+\frac{l^2r_k^{-2}}{2}}{\sqrt{2}})+&B(1,\frac{-lr_k^{-1}}{\sqrt{2}},\frac{lr_k^{-1}}{\sqrt{2}})+C(0,\frac{-1}{\sqrt{2}},\frac{1}{\sqrt{2}}):\\
&\frac{1}{2}r_{k+1}^{-1}\le |A|\lesssim \underline{C}_\e r_k^{-1},\quad |B|\lesssim \underline{C}_\e r_k^{-2},\quad |C|\lesssim  \underline{C}_\e R^{-1}\}.
\end{align*}
Defining $\w\in[\pi/4,\pi/2]$ by $(\cos\w,\sin\w)=(\frac{2\sqrt{2}lr_k^{-1}}{2+l^2r_k^{-2}},\frac{2-l^2r_k^{-2}}{2+l^2r_k^{-2}})$, the set $T(\tilde{\g}_k)$ is contained in  
\begin{align} \label{set2}
\{A(\cos\w,\sin\w,1)+&B(\sin\w,-\cos\w,0)+C(\cos\w,\sin\w,-1):\\
&r_{k+1}^{-1}\le |A|\lesssim \underline{C}_\e r_k^{-1},\quad |B|\lesssim \underline{C}_\e(r_k^{-2}+R^{-1}),\quad |C|\lesssim \underline{C}_\e R^{-1}\} \nonumber.
\end{align}

Suppose that $r_k^{-1}\le R^{-\frac{1}{2}}$. Then 
\begin{align} \label{set3}
T(\tilde{\g_k})\cap\{\xi_3\sim r^{-1}\}\subset\{A(\cos\w,\sin\w,1)+&B(\sin\w,-\cos\w,0)+C(\cos\w,\sin\w,-1):\\
&\frac{r^{-1}}{2}\le |A|\le r^{-1},\quad |B|\lesssim  \underline{C}_\e R^{-1},\quad |C|\lesssim  \underline{C}_\e R^{-1}\} \nonumber.
\end{align}
The $\w=\w(\g_k)$ in \eqref{set2} form an $\sim r_k^{-1}$-separated subset of $[\frac{\pi}{4},\frac{\pi}{2}]$. For a dyadic $S$ closest to $\underline{C}_\e R/r$, we may sort the $\w(\g_k)$ into different intervals $I_S\subset[0,2\pi)$ and note that the $r$ dilation of $T(\tilde{\g}_k)\cap\{\xi_3\sim r^{-1}\}$ for $\w(\g_k)\in I_S$ is contained in a single $\sim1\times S^{-1}\times S^{-1}$ conical small cap.

Now suppose that $R^{-\frac{1}{2}}\le r_k^{-1}\le R^{-\frac{1}{3}}$. Then 
\begin{align} \label{set4}
T(\tilde{\g_k})\cap\{\xi_3\sim r^{-1}\}\subset\{A(\cos\w,\sin\w,1)+&B(\sin\w,-\cos\w,0)+C(\cos\w,\sin\w,-1):\\
&\frac{r^{-1}}{2}\le |A|\le r^{-1},\quad |B|\lesssim \underline{C}_\e r_k^{-2},\quad |C|\lesssim  \underline{C}_\e R^{-1}\} \nonumber.
\end{align}
Let $S\in 2^\N$ be chosen so $S^{-\b_1}$ is the smallest dyadic number satisfying $\underline{C}_\e R^{\e}r_k^{-1}\le S^{-\b_1}$ (recalling that $\b_1$ is defined by $(Rr_k^{-1})^{-\b_1}=r_k^{-1}$ in the proposition statement). Then $\underline{C}_\e^{\b_1^{-1}}R^{\e\b_1^{-1}}r_kR^{-1}\le S$ and so each $r$-dilation of $T(\tilde{\g_k})\cap\{\xi_3\sim r^{-1}\}$ is contained in a single approximate $1\times S^{-\b_1}\times S^{-1}$ conical small cap. 

\end{proof}

To analyze $G_k^h$, we need to understand the Fourier support of $\sum_{\tau_k}|F_{\tau_k}^{k+1}|^2$ outside of a low set ${\mc{C}}_{R_{k+1}^{-\frac{1}{3}}}$. By \eqref{item3} of Lemma \ref{pruneprop}, the support of $\widehat{|F_{\g_k}^{k+1}|^2}$ is contained in  $\underline{C}_\e\tau_k-\underline{C}_\e\tau_k$. 

\begin{proposition} \label{geo3} Let $r$ be a dyadic value, $R_{k+1}^{-\frac{1}{3}}\le r^{-1}\le \underline{C}_\e R_k^{-\frac{1}{3}}$. There is an affine transformation $T:\R^3\to\R^3$ so that the following holds. We may partition the $\tau_k$ into $\lesssim_\e R^\e$ many sets $\mc{S}_i$ which satisfy: there is a canonical partition of the cone into approximate $1\times \underline{C}_\e rR_k^{-\frac{2}{3}}\times \underline{C}_\e^2r^2R_k^{-\frac{4}{3}}$ blocks so that for each $\tau_k\in\mc{S}_i$, the $r$-dilation of the sets $T[(\underline{C}_\e \tau_k-\underline{C}_\e\tau_k)\setminus B_{R_{k+1}^{-\frac{1}{3}}}]\cap\{\xi_3\sim r^{-1}\}$ is contained in one of the canonical cone blocks. \end{proposition}

\begin{proof} Suppose that $\tau_k$ is the $l$th piece, meaning that 
\[ \tau_k=\{(\xi_1,\xi_2,\xi_3): lR_k^{-\frac{1}{3}}\le \xi_1<(l+1)R_k^{-\frac{1}{3}},\,|\xi_2-\xi_1^2|\le R_k^{-\frac{2}{3}},\,|\xi_3-3\xi_1\xi_2+2\xi_1^3|\le R_k^{-1} \}   \]
where $l\in\{0,\ldots,R_k^{\frac{1}{3}}-1\}$. Let $T$ be the affine transformation from the proof of Proposition \ref{geo2}. Then $T[(\underline{C}_\e \tau_k-\underline{C}_\e\tau_k)\setminus B_{R_{k+1}^{-\frac{1}{3}}}]\cap\{\xi_3\sim r^{-1}\}$ is contained in the set 
\begin{align*}
\{A(\cos\w,\sin\w,1)+&B(\sin\w,-\cos\w,0)+C(\cos\w,\sin\w,-1):\\
&\frac{r^{-1}}{2}\le |A|\le r^{-1},\quad |B|\lesssim \underline{C}_\e R_k^{-\frac{2}{3}},\quad |C|\lesssim \underline{C}_\e R_k^{-1}\} \nonumber.
\end{align*}
where $\w\in[\frac{\pi}{4},\frac{\pi}{2}]$ is defined by $(\cos\w,\sin\w)=(\frac{2\sqrt{2}lR_k^{-\frac{1}{3}}}{2+l^2R_k^{-\frac{2}{3}}},\frac{2-l^2R_k^{-\frac{2}{3}}}{2+l^2R_k^{-\frac{2}{3}}})$. Since the $\w=\w(\tau_k)$ form an $\sim R_k^{-\frac{1}{3}}$-separated set, the $r$-dilation of each displayed set above is contained in a canonical cone block of approximate dimensions $1\times \underline{C}_\e rR_k^{-\frac{2}{3}}\times \underline{C}_\e^2r^2R_k^{-\frac{4}{3}}$.

\end{proof}

\subsection{Lemmas related to the high frequency parts of square functions \label{hisec}}

First we recall the small cap decoupling theorem for the cone from \cite{ampdep}. Subdivide the $R^{-1}$ neighborhood of the truncated cone $\Gamma=\{(\xi_1,\xi_2,\xi_3):\xi_1^2+\xi_2^2=\xi_3^2,\quad\frac{1}{2}\le \xi_3\le 1\}$ into $R^{-\b_2}\times R^{-\b_1}\times R^{-1}$ small caps $\g$, where $\b_1\in[\frac{1}{2},1]$ and $\b_2\in[0,1]$. Here, $R^{-\b_2}$ corresponds to the flat direction of the cone and $R^{-\b_1}$ corresponds to the angular direction. The $(\ell^p,L^p)$ small cap theorem for $\Gamma$ is the following. 
\begin{theorem}\label{conesmallcapthm}[Theorem 3 from \cite{ampdep}] Let $\b_1\in[\frac{1}{2},1]$ and $\b_2\in[0,1]$. For $p\ge 2$, 
\[ \int_{\R^3}|f|^p\le C_\e R^\e(R^{(\b_1+\b_2)(\frac{p}{2}-1)}+R^{(\b_1+\b_2)(p-2)-1}+R^{(\b_1+\b_2-\frac{1}{2})(p-2)})\sum_\g\|f_\g\|_{L^p(\R^3)}^p\]
for any Schwartz function $f:\R^3\to\C$ with Fourier transform supported in $\mc{N}_{R^{-1}}(\Gamma)$. 
\end{theorem}

\begin{lemma}[High lemma I]\label{high1} Suppose that $R^{-\b}\le r_k^{-1}\le R^{-\frac{1}{2}}$. Then
\[ \int|g_k^h|^{4} \le C_\e R^\e r_k^{-1}R\sum_{\z}\|\sum_{\g_k\subset\zeta}|f_{\g_k}^{k+1}|^2*\w_{\g_k}*\widecheck{\rho}_{>r_{k+1}^{-1}}\|_{L^4(\R^3)}^4 \]
where the $\z$ are disjoint collections of $r_k^2R^{-1}$ many adjacent $\g_k$. 
\end{lemma}

\begin{proof} Let $T$ be the affine transformation from Proposition \ref{geo2} and write $Tx=Ax+b$ for a $3\times 3$ invertible matrix $A$ and $b\in\R^3$. Then 
\begin{equation}\label{formg} g_k^h(x)=|\det A|^{-1}e^{-2\pi i x\cdot b}(\widehat{g_k^h}\circ T^{-1})^{\widecheck{\,\,\,}}((A^{-1})^* x) . \end{equation}
Perform the change of variables $x\mapsto A^*x$ to get
\[ \int|g_k^h(x)|^4dx=|\det A|^{-3}\int|(\widehat{g_k^h}\circ T^{-1})^{\widecheck{\,\,\,}}(x)|^4dx. \]
Let $r$ be a dyadic parameter in the range $r_{k+1}^{-1}\le r^{-1}\le\underline{C}_\e r_k^{-1}$. Let $\eta_r:\R^3\to[0,\infty)$ be a smooth function with compact support in the set $\{(\xi_1,\xi_2,\xi_3):\frac{r^{-1}}{2}\le \xi_3\le r^{-1}\}=:\{\xi_3\sim r^{-1}\}$ and satisfying the property that the sum of $\eta_r$ over dyadic $r$ is identically $1$ on the support of $\widehat{g_k^h}\circ T^{-1}$. By dyadic pigeonholing, there is an $r$ so that
\[ |\det A|^{-3}\int|(\widehat{g_k^h}\circ T^{-1})^{\widecheck{\,\,\,}}(x)|^4dx\le C_\e (\log R)^4|\det A|^{-3}\int|((\widehat{g_k^h}\circ T^{-1})\eta_r)^{\widecheck{\,\,\,}}(x)|^4dx.\]
Finally, perform the change of variables $x\mapsto rx$ to get
\[ |\det A|^{-3}r^3\int|((\widehat{g_k^h}\circ T^{-1})\eta_r)^{\widecheck{\,\,\,}}(rx)|^4dx. \]
Now, note that 
\begin{align*}  
((\widehat{g_k^h}\circ T^{-1})\eta_r)^{\widecheck{\,\,\,}}(rx)&=\sum_{\g_k}[(\widehat{|f_{\g_k}^{k+1}|^2}\widehat{\w_{\g_k}}(1-\rho_{\le r_{k+1}^{-1}}))\circ T^{-1}\cdot \eta_r]^{\widecheck{\,\,\,}}(rx)\\
&=\sum_i\sum_{\g_k\in\mc{S}_i}[(\widehat{|f_{\g_k}^{k+1}|^2}\widehat{\w_{\g_k}}(1-\rho_{\le r_{k+1}^{-1}}))\circ T^{-1}\cdot\eta_r]^{\widecheck{\,\,\,}}(rx) \end{align*} 
where $\mc{S}_i$ is one of the $\lesssim_\e R^\e$ many sets partitioning the $\g_k$ from (1) of Proposition \ref{geo2}. Apply the triangle inequality in the first sum over $i$ and then apply Theorem \ref{conesmallcapthm} with parameters $\underline{C}_\e^{-1}\frac{R}{r}$, $\b_1=1$, and $\b_2=0$ to obtain 
\begin{align*}
    \int|g_k^h|^4&\lesssim_\e (\log R)R^{6\e} (r_k^{-1}R)|\det A|^{-3}r^3\sum_{\zeta}\int |\sum_{\g_k\subset\z}[(\widehat{|f_{\g_k}^{k+1}|^2}\widehat{\w_{\g_k}}(1-\rho_{\le r_{k+1}^{-1}}))\circ T^{-1}\cdot\eta_r]^{\widecheck{\,\,\,}}(rx)|^4dx
\end{align*}
where $\zeta$ are disjoint collections of $\sim r_k^{-2}R$ many neighboring $\g_k$. It remains to undo the initial steps which allowed us to apply small cap decoupling for the cone. First do the change of variables $x\mapsto r^{-1}x$:
\[ r^3\sum_{\zeta}\int |\sum_{\g_k\subset\z}[(\widehat{|f_{\g_k}^{k+1}|^2}\widehat{\w_{\g_k}}(1-\rho_{\le r_{k+1}^{-1}}))\circ T^{-1}\cdot\eta_r]^{\widecheck{\,\,\,}}(rx)|^4dx=\sum_{\zeta}\int |\sum_{\g_k\subset\z}[(\widehat{|f_{\g_k}^{k+1}|^2}\widehat{\w_{\g_k}}(1-\rho_{\le r_{k+1}^{-1}}))\circ T^{-1}\cdot\eta_r]^{\widecheck{\,\,\,}}(x)|^4dx. \]
By Young's convolution inequality (since multiplication on the Fourier side by $\eta_r$ is equivalent to convolution on the spatial side by $\widecheck{\eta}_r$, which is $L^1$-normalized),
\[ \sum_{\zeta}\int |\sum_{\g_k\subset\z}[(\widehat{|f_{\g_k}^{k+1}|^2}\widehat{\w_{\g_k}}(1-\rho_{\le r_{k+1}^{-1}}))\circ T^{-1}\cdot\eta_r]^{\widecheck{\,\,\,}}|^4 \lesssim \sum_{\zeta}\int |\sum_{\g_k\subset\z}[(\widehat{|f_{\g_k}^{k+1}|^2}\widehat{\w_{\g_k}}(1-\rho_{\le r_{k+1}^{-1}}))\circ T^{-1}]^{\widecheck{\,\,\,}}|^4. \]
Perform the change of variables $x\mapsto (A^{-1})^{*}x$ and use \eqref{formg} to get
\[ |\det A|^3\sum_{\zeta}\int |\sum_{\g_k\subset\z}[(\widehat{|f_{\g_k}^{k+1}|^2}\widehat{\w_{\g_k}}(1-\rho_{\le r_{k+1}^{-1}}))\circ T^{-1}]^{\widecheck{\,\,\,}}|^4\lesssim \sum_{\zeta}\int |\sum_{\g_k\subset\z}[\widehat{|f_{\g_k}^{k+1}|^2}\widehat{\w_{\g_k}}(1-\rho_{\le r_{k+1}^{-1}})]^{\widecheck{\,\,\,}}|^4, \]
which finishes the proof. 

\end{proof}

\begin{lemma}[High lemma II]\label{high2} Suppose that $\max(R^{-\b},R^{-\frac{1}{2}})\le r_k^{-1}\le R^{-\frac{1}{3}}$. Then
\[ \int|g_k^h|^{2+\frac{2}{\b_1}} \le C_\e R^{14\e} r_kR\sum_{\g_k}\|f_{\g_k}^{k+1}\|_{L^{4+\frac{4}{\b_1}}(\R^3)}^{4+\frac{4}{\b_1}} \]
where $\b_1\in[\frac{1}{2},1]$ satisfies $(r_kR)^{-\b_1}=r_k$. \end{lemma}

\begin{proof} Repeat the argument from the proof of Lemma \ref{high1}, using (2) in place of (1) from Proposition \ref{geo2} and applying Theorem \ref{conesmallcapthm} with $\b_1$ as in the hypothesis of the lemma and $\b_2=0$. The result is 
\[ \int|g_k^h|^{2+\frac{2}{\b_1}}\lesssim_\e R^{14\e} \sum_{\g_k}\int |[\widehat{|f_{\g_k}^{k+1}|^2}\widehat{\w_{\g_k}}(1-\rho_{\le r_{k+1}^{-1}})]^{\widecheck{\,\,\,}}|^{2+\frac{2}{\b_1}} .\]
Since $1-\rho_{\le r_{k+1}^{-1}}=\rho_{\le \underline{C}_\e}-\rho_{\le r_{k+1}^{-1}}$ on the support of $\widehat{|f_{\g_k}^{k+1}|^2}$, by Young's convolution inequality, we have 
\[\int |[\widehat{|f_{\g_k}^{k+1}|^2}\widehat{\w_{\g_k}}(1-\rho_{\le r_{k+1}^{-1}})]^{\widecheck{\,\,\,}}|^{2+\frac{2}{\b_1}}\lesssim \int |(\widehat{|f_{\g_k}^{k+1}|^2})^{\widecheck{\,\,\,}}|^{2+\frac{2}{\b_1}}=\int |f_{\g_k}^{k+1}|^{4+\frac{4}{\b_1}} .\]

\end{proof}

\begin{lemma}\label{high3} For each $m$, $1\le m\le N$, 
\[ \int|G_m^h|^6\le C_\e R^\e \big(\sum_{\tau_m}\|F_{\tau_m}^{m+1}\|_{L^{12}(\R^3)}^4\big)^3. \]
\end{lemma}
\begin{proof}
Repeat the argument from the proof of Lemma \ref{high1}, using Proposition \ref{geo3} in place of Proposition \ref{geo2} and applying canonical $L^6$ cone decoupling \cite{BD} instead of small cap decoupling. The result is 
\[ \int|G_m^h|^{6}\lesssim_\e R^{8\e} \sum_{\tau_m}\int |[\widehat{|F_{\tau_m}^{m+1}|^2}\widehat{\w_{\tau_m}}(1-\rho_{\le R_{m+1}^{-\frac{1}{3}}})]^{\widecheck{\,\,\,}}|^{6} \]
Since $1-\rho_{\le R_{m+1}^{-\frac{1}{3}}}=\rho_{\le \underline{C}_\e}-\rho_{\le R_{m+1}^{-\frac{1}{3}}}$ on the support of $\widehat{|F_{\tau_m}^{m+1}|^2}$, by Young's convolution inequality, we have 
\[\int |[\widehat{|F_{\tau_m}^{m+1}|^2}\widehat{\w_{\tau_m}}(1-\rho_{\le R_{m+1}^{-\frac{1}{3}}})]^{\widecheck{\,\,\,}}|^{6}\lesssim \int |(\widehat{|F_{\tau_m}^{m+1}|^2})^{\widecheck{\,\,\,}}|^{6}=\int |F_{\tau_m}^{m+1}|^{12} .\]

\end{proof}

\begin{theorem}[Cylindrical Decoupling over $\P^1$] \label{cyldec}
Let $\mb{P}^1=\{(t,t^2):0\le t\le 1\}$ and for $\d>0$, let $\mc{N}_\d(\P^1)$ denote the $\d$-neighborhood of $\mc{P}^1$ in $\R^2$. If $h:\R^3\to\C$ is a Schwartz function with Fourier transform supported in $\mc{N}_\d(\P^1)\times\R$, then for each $4\le p\le 6$,
\[ \int_{\R^3}|h|^p\lesssim_\e d^{-\e}(\sum_{\z}\|h_\z\|_{L^p(\R^3)}^2)^\frac{p}{2}\]
where the $\z$ are products of $\sim \d^{1/2}\times\d$ rectangles that partition $\mc{N}_\d(\P^1)$ with $\R$.
\end{theorem}

\begin{proof} Begin by using Fourier inversion to write
\[ h(x',x_3)=\int_{\mc{N}_\d(\P^1)}\int_{\R}\widehat{h}(\xi',\xi_3)e^{2\pi i \xi\cdot x'}e^{2\pi i \xi_3x_3} d\xi_3d\xi' . \]
For each $x_3$, the function $x'\mapsto \int_{\mc{N}_\d(\P^1)}\int_\R\widehat{h}(\xi',\xi_3)e^{2\pi i \xi_3x_3} d\xi_3e^{2\pi i \xi\cdot x'}d\xi'$ satisfies the hypotheses of the decoupling theorem for $\P^1$. Use Fubini's theorem to apply the $\ell^2$ decoupling theorem for $\P^1$ from \cite{BD} to the inner integral
\[  \int_{\R}\int_{\R^2}|h(x',x_3)|^pdx'dx_3\lesssim_\e\int_{\R}\d^{-\e}\big(\sum_\nu\big(\int_{\R^2}\big|\int_{\nu}\int_{\R}\widehat{h}(\xi',\xi_3)e^{2\pi i \xi\cdot x'}e^{2\pi i \xi_3x_3} d\xi_3d\xi' \big|^pdx'\big)^{\frac{2}{p}}\big)^{\frac{p}{2}}  dx_3\]
where $\{\nu\}$ form a partition of $\mc{N}_\d(\P^1)$ into $\sim\d^{1/2}\times \d$ blocks. By the triangle inequality, the right hand side above (omitting $C_\e \d^{-\e}$) is bounded by 
\[ \big(\sum_\nu\big(\int_{\R}\int_{\R^2}\big|\int_{\nu}\int_{\R}\widehat{h}(\xi',\xi_3)e^{2\pi i \xi\cdot x'}e^{2\pi i \xi_3x_3} d\xi_3d\xi' \big|^pdx'dx_3\big)^{\frac{2}{p}}\big)^{\frac{p}{2}} . \]
The sets $\nu\times\R$ are the $\z$ in the statement of the lemma. 

\end{proof}

\begin{rmk}
The implicit upper bound in the statement of Theorem \ref{cyldec} is uniform in $4\le p\le 6$. For the specific exponent $p=4$, the implicit $C_\e \d^{-\e}$ upper bound may be replaced by an absolute constant $B$ which does not depend on $\d$. 
\end{rmk}

\subsection{Local trilinear restriction for $\mc{M}^3$ \label{trirestpf}}

The weight function $W_{B_r}$ in the following theorem decays by a factor of $10$ off of the ball $B_r$. It is specifically defined in Definition \ref{M3ballweight}. 
\begin{prop}\label{trirestprop}
Let $s\ge 10r\ge10$ and let $f:\R^3\to\C$ be a Schwartz function with Fourier transform supported in $\mc{N}_{r^{-1}}(\mc{M}^3)$. Suppose that $\tau_1^1,\tau_1^2,\tau_1^3$ are canonical moment curve blocks at scale $R_1^{\frac{1}{3}}$ which satisfy $\text{dist}(\tau_1^i,\tau_1^j)\ge s^{-1}$ for $i\not=j$. Then 
\[  \int_{B_r}|f_{\tau_1^1}f_{\tau_1^2}f_{\tau_1^3}|^2\lesssim s^3|B_r|^{-2}\big(\int|f_{\tau_1^1}|^2W_{B_r}\big)\big(\int|f_{\tau_1^2}|^2W_{B_r}\big)\big(\int|f_{\tau_1^3}|^2W_{B_r}\big) . \]
\end{prop}
\noindent The weight function $W_{B_r}$ is the generic ball weight defined in Definition \ref{M3ballweight}.

\begin{proof} Let $\g(t)=(t,t^2,t^3)$ and let $B_{r^{-1}}$ be the ball of radius $r^{-1}$ in $\R^3$ centered at the origin. Then 
\begin{align*}
    W_{B_r}(x)f_{\tau_1^i}(x)&= \int_{\tau_1^i+B_{r^{-1}}}\widehat{W}_{B_r}*\widehat{f}_{\tau_1^i}(\xi^i)e^{2\pi i x\cdot \xi^i}d\xi^i\\
    &= \int_{\tau_1^i+B_{r^{-1}}}\widehat{W}_{B_r}*\widehat{f}_{\tau_1^i}(\xi_1^i,\xi_2^i,\xi_3^i)e^{2\pi i x\cdot (\xi_1^i,\xi_2^i,\xi_2^i)}d\xi_1^i\xi_2^i\xi_3^i \\
    &= \int_{|\{\w_i\in\R^2:|\w_i|\le 2r^{-1}\}}\int_{I_i}\widehat{W}_{B_r}*\widehat{f}_{\tau_1^i}(\g(\xi_1^i)+(0,\w_i))e^{2\pi i x\cdot(\g(\xi_1^i)+(0,\w_i))}d\xi_1^id\w^i. 
\end{align*}
where $B_{r^{-1}}+\supp f_{\tau_1^i}\subset\{\g(\xi_1^i)+(0,\w_i):\xi_1^i\in I_1,\quad|\w_i|\le r^{-1}\}$. Let $\{\w_i\in\R^2:|\w_i|\le 2r^{-1}\}=B_{r^{-1}}^{(2)}$. Then for $\w=(\w_1,\w_2,\w_3)$, we have
\begin{align}
\nonumber    \int|W_{B_r}(x)f_{\tau_1^1}(x)W_{B_r}(x)&f_{\tau_1^2}(x)W_{B_r}(x)f_{\tau_1^3}(x)|^2dx\\
\nonumber    &= \int_{B_r}\left|\prod_{i=1}^3\int_{B_{r^{-1}}^{(2)}}\int_{I_i}\widehat{W}_{B_r}*\widehat{f}_{\tau_1^i}(\g(\xi_1^i)+(0,\w_i))e^{2\pi i x\cdot(\g(\xi_1^i)+(0,\w_i))}d\xi_1^id\w_i\right|^2dx \\
\nonumber    &\le \int_{B_r}\left|\int_{(B_{r^{-1}}^{(2)})^3}\left|\prod_{i=1}^3\int_{I_i}\widehat{W}_{B_r}*\widehat{f}_{\tau_1^i}(\g(\xi_1^i)+(0,\w_i))e^{2\pi i x\cdot\g(\xi_1^i)}d\xi_1^i\right|d\w\right|^2dx \\
\label{mainline}    &\le \left(\int_{(B_{r^{-1}}^{(2)})^3}\left(\int_{B_r}\left|\prod_{i=1}^3\int_{I_i}\widehat{W}_{B_r}*\widehat{f}_{\tau_1^i}(\g(\xi_1^i)+(0,\w_i))e^{2\pi i x\cdot\g(\xi_1^i)}d\xi_1^i\right|^2dx\right)^{1/2}d\w \right)^2. 
\end{align}
For each $\w\in (B_{r^{-1}}^{(2)})^3$, analyze the inner integral in $x$. Use the abbreviation $\widehat{W}_{B_r}*\widehat{f}_{\tau_1^i}(\cdot+(0,\w_i))=\widehat{f}_{\tau_1^i}^{\w_i}(\cdot)$ and further manipulate the innermost integral as a function of $x$:
\begin{align*} 
\prod_{i=1}^3\int_{I_i}\widehat{W}_{B_r}*&\widehat{f}_{\tau_i}(\g(\xi_1^i)+(0,\w_i))e^{2\pi i x\cdot\g(\xi_1^i)}d\xi_1^i\\
&=\int_{I_1\times I_2\times I_3}\widehat{f}_{\tau_1^1}^{\w_1}(\g(\xi_1^1))\widehat{f}_{\tau_1^2}^{\w_2}(\g(\xi_1^2))\widehat{f}_{\tau_1^3}^{\w_3}(\g(\xi_1^3))e^{2\pi i x\cdot[\g(\xi_1^i)+\g(\xi_1^2)+\g(\xi_1^3)]}d\xi_1
\end{align*}
where $\xi_1=(\xi_1^1,\xi_1^2,\xi_1^3)$. Perform the change of variables $\tilde{\xi}=\g(\xi_1^1)+\g(\xi_1^2)+\g(\xi_1^3)$. The Jacobian factor is $\frac{1}{|\det J|}$ where $\det J$ is defined explicitly in terms of $\xi_1$ by
\begin{align*}
    \det \begin{bmatrix}
    1&1&1\\
    2\xi_1^1&2\xi_1^2&2\xi_1^3 \\
    3(\xi_1^1)^2&3(\xi_1^2)^2&3(\xi_1^3)^2 
    \end{bmatrix}&= 6(\xi_2-\xi_1)(\xi_3-\xi_1)(\xi_3-\xi_2),
\end{align*}
using the formula for the determinant of a Vandermonde matrix. Note that since $\text{dist}(I_i,I_j)\ge s^{-1}-2r^{-1}>0$, $|\det J|$ is nonzero. The change of variables yields 
\begin{align}\label{inthere}
    \int_{\g(I_1)+\g(I_2)+\g(I_3)} \widehat{f}_{\tau_1^1}^{\w_1}(\g(\xi_1^1))\widehat{f}_{\tau_1^2}^{\w_2}(\g(\xi_1^2))\widehat{f}_{\tau_1^3}^{\w_3}(\g(\xi_1^3))e^{2\pi i x\cdot\tilde{\xi}}\frac{1}{|\det J(\xi_1)|}d\tilde{\xi}
\end{align}
where we interpret $\xi_1$ in the integrand as implicitly depending on $\tilde{\xi}$.
Define $F^\w(\tilde{\xi})$ by 
\[ \chi_{\g(I_1)+\g(I_2)+\g(I_3)}(\tilde{\xi})\widehat{f}_{\tau_1^1}^{\w_1}(\g(\xi_1^1))\widehat{f}_{\tau_1^2}^{\w_2}(\g(\xi_1^2))\widehat{f}_{\tau_1^3}^{\w_3}(\g(\xi_1^3))\frac{1}{|\det J(\xi_1)|} \]
so that we may view the integral in \eqref{inthere} as the inverse Fourier transform of $F^\w$. The summary of the inequality so far, picking up from \eqref{mainline}  and using the change of variables and the definition of $F^\w$, is
\begin{align*}
\int_{B_r}|f_{\tau_1^1}(x)f_{\tau_1^2}(x)f_{\tau_1^3}(x)|^2dx&\lesssim \left(\int_{(B_{r^{-1}}^{(2)})^3}\left(\int\left|\widecheck{F}^\w(x)\right|^2dx\right)^{1/2}d\w \right)^2.
\end{align*}
By Plancherel's theorem, the right hand side above equals
\[\left(\int_{(B_{r^{-1}}^{(2)})^3}\left(\int\left|{F}^\w(\tilde{\xi})\right|^2d\tilde{\xi}\right)^{1/2}d\w \right)^2. \]
By Cauchy-Schwarz, this is bounded above by 
\[ |(B_{r^{-1}}^{(2)})^3|\int_{(B_{r^{-1}}^{(2)})^3}\int\left|{F}^\w(\tilde{\xi})\right|^2d\tilde{\xi}d\w\sim r^{-6}\int_{(B_{r^{-1}}^{(2)})^3}\int\left|{F}^\w(\tilde{\xi})\right|^2d\tilde{\xi}d\w. \]
Undo the change of variables, again writing $\tilde{\xi}=\g(\xi_1^1)+\g(\xi_1^2)+\g(\xi_1^3)$ 
to get
\[ r^{-6}\int_{(B_{r^{-1}}^{(2)})^3}\int_{I_1\times I_2\times I_3}\left|\widehat{f}_{\tau_1^1}^{\w_1}(\g(\xi_1^1))\widehat{f}_{\tau_1^2}^{\w_2}(\g(\xi_1^2))\widehat{f}_{\tau_1^3}^{\w_3}(\g(\xi_1^3))\right|^2|\det J(\xi_1)|^{-1}d{\xi}_1d\w  . \]
Note that $|\det J(\xi_1)|\gtrsim s^{-3}$, so the previous line is bounded by 
\[ r^{-6}s^3\int_{(B_{r^{-1}}^{(2)})^3}\int\left|\widehat{f}_{\tau_1^1}^{\w_1}(\g(\xi_1^1))\widehat{f}_{\tau_1^2}^{\w_2}(\g(\xi_1^2))\widehat{f}_{\tau_1^3}^{\w_3}(\g(\xi_1^3))\right|^2d{\xi}_1d\w \sim r^{-6}s^3\prod_{i=1}^3\int_{\mc{N}_{r^{-1}}(\tau_i)}\left|\widehat{W}_{B_r}*\widehat{f}_{\tau_1^i}(\xi)\right|^2d{\xi} . \]
By Plancherel's Theorem, this is bounded by
\[  r^{-6}s^3\prod_{i=1}^3\int_{\R^3}\left|f_{\tau_1^i}(x)\right|^2 W_{B_r}dx . \]

\end{proof}

\section{A weak version of Theorem \ref{main} for the critical exponent \label{alphapc}}

\subsection{The broad part of $U_\a$ \label{broad}}

For three canonical blocks $\tau_1^1,\tau_1^2,\tau_1^3$ (with dimensions $\sim R_1^{-\frac{1}{3}}\times R_1^{-\frac{2}{3}}\times R_1^{-1}$) which are pairwise $\ge 10\underline{C}_\e R^{-\frac{\e}{3}}$-separated, where $\underline{C}_\e$ is from Lemma \ref{pruneprop}, define the broad part of $U_\a$ to be
\[ \text{Br}_\a^K=\{x\in U_\a: \a\le K|f_{\tau_1^1}(x)f_{\tau_1^2}(x)f_{\tau_1^3}(x)|^{\frac{1}{3}},\quad\max_{\tau_1^i}|f_{\tau_1^i}(x)|\le \a\}. \]

We bound the broad part of $U_\a$ in the following proposition. 

\begin{prop}\label{mainprop} Let $R,K\ge1$. Suppose that $\|f_\g\|_{L^\infty(\R^3)}\le 2$ for all $\g$. Then \begin{equation*} 
\a^{6+\frac{2}{\b}}|\text{\emph{Br}}_\a^{K}|\lesssim_\e  K^{50}R^{10\e}A_\e^{10(M+N)} R^{2\b+1}\sum_{\g}\|f_\g\|_{L^2(\R^3)}^2 .
\end{equation*} 
\end{prop}

\begin{proof}[Proof of Proposition \ref{mainprop}]
Begin by observing that we may assume that $R^\b\le \a^2$. Indeed, if $\a^2\le R^\b$, then we have 
\[ \a^{6+\frac{2}{\b}}|U_\a|\le R^{2\b+1}\|f\|_{L^2(\R^3)}^2\le R^{2\b+1}\sum_\g\|f_\g\|_2^2 \]
using $L^2$-orthogonality. Assume for the remainder of the argument that $R^\b\le\a^2$.


We bound each of the sets $\text{Br}_\a^K\cap\Lambda_k$, $\text{Br}_\a^K\cap\Omega_m$, and $\text{Br}_\a^K\cap L$ in separate cases. It suffices to consider the case that $R$ is at least some constant depending on $\e$ since if $R\le C_\e$, we may prove the proposition using trivial inequalities. 

\vspace{2mm}
\noindent\underline{Case 1: bounding $| \text{Br}_\a^{K}\cap\Lambda_k|$. } 
By Lemma \ref{ftofk}, 
\[ |\text{Br}_\a^K\cap\Lambda_k|\le|\{x\in  U_\a\cap\Lambda_k:\a\lesssim K|f_{\tau_1^1}^{k+1}(x)f_{\tau_1^2}^{k+1}(x)f_{\tau_1^3}^{k+1}(x)|^{\frac{1}{3}},\quad\max_{\tau_1^i}|f_{\tau_1^i}(x)|\le \a\}.\]
By Lemma \ref{pruneprop}, the Fourier supports of $f_{\tau_1^1}^{k+1},f_{\tau_1^2}^{k+1},f_{\tau_1^3}^{k+1}$ are contained in the $\underline{C}_\e r_k^{-1}$-neighborhood of $\underline{C_\e} \tau_1^1,\underline{C}_\e \tau_1^2,\underline{C}_\e\tau_1^3$ respectively, which are $\ge\underline{C}_\e R^{-\frac{\e}{3}}$-separated blocks of the moment curve. Let $\{B_{r_k}\}$ be a finitely overlapping cover of $\text{Br}_\a^{K}\cap\Lambda_k$ by $r_k$-balls. For $R$ large enough depending on $\e$, apply Proposition \ref{trirestprop} to get
\begin{align*}
    \int_{B_{r_k}}|f_{\tau_1^1}^{k+1}f_{\tau_1^2}^{k+1}f_{\tau_1^3}^{k+1}|^2&\lesssim_\e R^{\e} |B_{r_k}|^{-2}\Big(\int|f_{\tau_1^1}^{k+1}|^2W_{B_{r_k}}\Big)\Big(\int|f_{\tau_1^2}^{k+1}|^2W_{B_{r_k}}\Big)\Big(\int|f_{\tau_1^3}^{k+1}|^2W_{B_{r_k}}\Big).
\end{align*}
Using local $L^2$-orthogonality (Lemma \ref{L2orth}), each integral on the right hand side above is bounded by 
\[ C_\e \int\sum_{\tau_k}|f_{\g_k}^{k+1}|^2W_{B_{r_k}}. \]
If $x\in \text{Br}_\a^{K}\cap\Lambda_k\cap B_{r_k}$, then the above integral is bounded by 
\[ C_\e \int \sum_{\g_k}|f_{\g_k}^{k+1}|^2*\w_{\g_k}W_{B_{r_k}}\lesssim C_\e |B_{r_k}| \sum_{\g_k}|f_{\g_k}^{k+1}|^2*\w_{\g_k}(x) \]
by the locally constant property (Lemma \ref{locconst}) and properties of the weight functions. The summary of the inequalities so far is that 
\[ \a^6|\text{Br}_\a^{K}\cap\Lambda_k\cap B_{r_k}|\lesssim_\e K^6\int_{B_{r_k}}|f_{\tau_1^1}^{k+1}f_{\tau_1^2}^{k+1}f_{\tau_1^3}^{k+1}|^2\lesssim_\e R^\e K^6 |B_{r_k}|g_k(x)^3 \]
where $x\in \text{Br}_\a^{K}\cap\Lambda_k\cap B_{r_k}$. 

Recall that since $x\in\Lambda_k$, we have the lower bound $A_\e^{M-k}R^\b\le g_k(x)$ (where $A_\e$ is from Definition \ref{impsets}), which leads to the inequality
\[ \a^6|\text{Br}_\a^{K}\cap\Lambda_k\cap B_{r_k}|\lesssim_\e K^6 R^{\e} \frac{1}{(A_\e^{M-k}R^\b)^p}|B_{r_k}|g_k(x)^{3+p} \]
for any $p\ge 0$. By Corollary \ref{highdom}, we also have the upper bound $|g_k(x)|\le 2|g_k^h(x)|$, so that 
\[ \a^6|\text{Br}_\a^{K}\cap\Lambda_k\cap B_{r_k}|\lesssim_\e K^6 R^{\e} \frac{1}{(A_\e^{M-k}R^\b)^p}|B_{r_k}||g_k^h(x)|^{3+p} \]
for any $p\ge 0$. By the locally constant property applied to $g_k^h$, $|g_k^h|^{3+p}\lesssim_\e |g_k^h*w_{ B_{r_k}}|^{3+p}$ and by Cauchy-Schwarz, $|g_k^h*w_{ B_{r_k}}|^{3+p}\lesssim |g_k^h|^{3+p}*w_{B_{r_k}}$. Combine this with the previous displayed inequality to get 
\[ \a^6|\text{Br}_\a^{K}\cap\Lambda_k\cap B_{r_k}|\lesssim_\e K^6 R^{\e} \frac{1}{(A_\e^{M-k}R^\b)^p}\int|g_k^h|^{3+p}W_{ B_{r_k}} .\]
Summing over the balls $B_{r_k}$ in our finitely-overlapping cover of $\text{Br}_\a^{K}\cap\Lambda_k$, we conclude that
\begin{equation}\label{peqn} \a^6|\text{Br}_\a^{K}\cap\Lambda_k|\lesssim_\e K^6 R^{\e} \frac{1}{(A_\e^{M-k} R^\b)^{p}}\int_{\R^3}|g_k^h|^{3+p} .\end{equation}
We are done using the properties of the set $\text{Br}_\a^{K}\cap\Lambda_k$, which is why we now integrate over all of $\R^3$ on the right hand side. We will choose different $p>0$ and analyze the high part $g_k^h$ in two sub-cases which depend on the size of $r_k$.

\vspace{2mm}
\noindent\underline{Subcase 1a: $R^{-\b}\le r_k^{-1}\le R^{-\frac{1}{2}}$.}
This case only appears if $\frac{1}{2}\le \b$. Choose $p=1$ in \eqref{peqn} and use Lemma \ref{high1} to obtain
\[ \a^6|\text{Br}_\a^{K}\cap\Lambda_k|\lesssim_\e K^6 R^{\e}   \frac{1}{A_\e^{M-k}R^\b} C_\e R^\e r_k^{-1}R\sum_{\z}\|\sum_{\g_k\subset\z}|f_{\g_k}^{k+1}|^2*\w_{\g_k}*\widecheck{\rho}_{>r_{k+1}^{-1}}\|_{L^4(\R^3)}^4 \]
where $\z$ are collections of $r_k^2R^{-1}$ many adjacent $\g_k$. 

The Fourier supports of the terms in the $L^4$ norm are still approximately disjoint (actually $C_\e$-overlapping, see Proposition \ref{geo1}), so by Plancherel's theorem and $L^2$-orthogonality, we have
\begin{align}  
\|\sum_{\g_k\subset\z}|f_{\g_k}^{k+1}|^2*&\w_{\g_k}*\widecheck{\rho}_{>r_{k+1}^{-1}}\|_{L^4(\R^3)}^4 \nonumber\\
&\lesssim_\e R^\e \|\sum_{\g_k\subset\z}|f_{\g_k}^{k+1}|^2*\w_{\g_k}*\widecheck{\rho}_{>r_{k+1}^{-1}}\|_{L^\infty(\R^3)}^2\sum_{\g_k\subset\z}\||f_{\g_k}^{k+1}|^2*\w_{\g_k}*\widecheck{\rho}_{>r_{k+1}^{-1}}\|_{L^2(\R^3)}^2 \label{upp}
\end{align}
for each $\z$. First bound the $L^\infty$ norm by
\[ \|\sum_{\g_k\subset\z}|f_{\g_k}^{k+1}|^2*\w_{\g_k}*\widecheck{\rho}_{>r_{k+1}^{-1}}\|_{L^\infty(\R^3)}^2\lesssim (\#\g_k\subset\z)^2\max_{\g_k}\|f_{\g_k}^{k+1}\|_{L^\infty(\R^3)}^4\lesssim (r_k^2R^{-1})^2\max_{\g_k}\|f_{\g_k}^{k+1}\|_{L^\infty(\R^3)}^4\]
where we used that $\|\w_k*\widecheck{\rho}_{>r_{k+1}^{-1}}\|_1\sim 1$. 
To bound each of the $L^2$ norms in \eqref{upp}, we use cylindrical $L^4$-decoupling the parabola (Theorem \ref{cyldec}) and unravel the pruning process using properties from Lemma \ref{pruneprop}:
\begin{align*}
(\text{Young's inequality})\qquad\qquad  \quad \||f_{\g_k}^{k+1}|^2*&\w_{\g_k}*\widecheck{\rho}_{>r_{k+1}^{-1}}\|_{L^2(\R^3)}^2\lesssim \|f_{\g_k}^{k+1}\|_{L^4(\R^3)}^4\\
(\text{$L^4$ cyl. dec. for $\P^1$})\quad\qquad\qquad\qquad\qquad &\lesssim_\e R^{\e^2} \big(\sum_{\g_{k+1}\subset\g_k}\|f_{\g_{k+1}}^{k+1}\|_{L^4(\R^3)}^2\big)^2\\
(\text{\eqref{item1} from Lemma \ref{pruneprop}})\qquad\qquad\qquad\qquad\quad &\lesssim \big(\sum_{\g_{k+1}\subset\g_k}\|f_{\g_{k+1}}^{k+2}\|_{L^4(\R^3)}^2\big)^2\\
(\text{iterate previous two inequalities})\quad\qquad\qquad &\lesssim \cdots\lesssim  \big(\sum_{\g_N\subset\g_k}\|f_{\g_{N}}^{N}\|_{L^4(\R^3)}^2\big)^2\lesssim \big(\sum_{\g\subset\g_k}\|f_{\g}\|_{L^4(\R^3)}^2\big)^2.
\end{align*} 
Note that each application of $L^4$-decoupling involved an explicit constant $B$ in the upper bound, so does not depend on a scale $R$. The accumulated constant in the unwinding-the-pruning process above is $B^{C\e^{-1}}$ since there are fewer than $\sim \e^{-1}$ many different scales of $\g_k$ until we arrive at $\g$. Use Cauchy-Schwarz to bound the expression in the final upper bound above by 
\[ \#\g\subset\g_k\sum_{\g\subset\g_k}\|f_{\g}\|_{L^4(\R^3)}^4\lesssim (r_k^{-1}R^\b)\sum_{\g\subset\g_k}\|f_{\g}\|_{L^4(\R^3)}^4. \]
Using the assumption $\|f_\g\|_\infty\lesssim 1$ for each $\g$, $\|f_\g\|_{L^4(\R^4)}^4\lesssim \|f_{\g}\|_{L^2(\R^3)}^2$. The summary of the argument in this case so far is that
\begin{align*} 
\a^6|\text{Br}_\a^{K}\cap\Lambda_k|&\lesssim_\e K^6 R^{2\e} R^{-\b} r_k^{-1}R\sum_{\z}(r_k^2R^{-1})^2\max_{\g_k}\|f_{\g_k}^{k+1}\|_{\infty}^4(r_k^{-1}R^\b)\sum_{\g_k\subset\z}\|f_{\g}\|_2^2 \\
    &\lesssim_\e K^6 R^{2\e} r_k^{2}R^{-1}\max_{\g_k}\|f_{\g_k}^{k+1}\|_{\infty}^4\sum_{\g}\|f_{\g}\|_2^2 .
\end{align*}

It now suffices to verify that $r_k^2R^{-1}\max_{\g_k}\|f_{\g_k}^{k+1}\|_\infty^4\lessapprox R^{2\b+1}\a^{-\frac{2}{\b}}$. We will use the upper bounds $\|f_{\g_k}^{k+1}\|_\infty\lesssim \min(r_k^{-1}R^\b,K^3A_\e^{M-k}\frac{R^\b}{\a})$ (from \eqref{item1} and \eqref{item2} in Lemma \ref{pruneprop}). Suppose that $r_k<\a$. Use $\|f_{\g_k}^{k+1}\|_\infty\lesssim K^3 A_\e^{M-k}\frac{R^\b}{\a}$ and $\b\ge \frac{1}{2}$ to check:
\begin{align*}
(r_k)^{\frac{2}{\b}-2}{\le}(R^\b)^{\frac{2}{\b}-2}\qquad&\implies\qquad 
r_k^2 R^{-1+4\b}{\le}  R^{2\b+1}(r_k^{-1})^{4-\frac{2}{\b}}\\ 
&\implies \qquad r_k^2R^{-1}\big(\frac{R^{\b}}{\a}\big)^4{\le} R^{2\b+1} \a^{-\frac{2}{\b}} \\
&\implies \qquad r_k^2R^{-1}\max_{\g_k}\|f_{\g_k}^{k+1}\|_\infty^4{\lesssim} A_\e^{4(M-k)}R^{2\b+1} \a^{-\frac{2}{\b}} ,
\end{align*}
as desired. Now suppose that $r_k\ge \a$. Then use $\|f_{\g_k}^{k+1}\|_\infty\lesssim r_k^{-1}R^\b$ and check:
\begin{align*}
(r_k)^{\frac{2}{\b}-2}\le (R^\b)^{\frac{2}{\b}-2}\qquad&\implies\qquad r_k^2R^{-1}(r_k^{-1}R^\b)^4\le R^{2\b+1}(r_k)^{-\frac{2}{\b}}\\
&\implies \qquad r_k^2R^{-1}\max_{\g_k}\|f_{\g_k}^{k+1}\|_\infty^4\lesssim R^{2\b+1}(\a)^{-\frac{2}{\b}}, 
\end{align*}
which finishes this subcase.

\vspace{2mm}
\noindent\underline{Subcase 1b: $\max(R^{-\b},R^{-\frac{1}{2}})\le r_k^{-1}\le R^{-\frac{1}{3}}$. } In this case, let $\b_1\in[\frac{1}{2},1]$ satisfy $(r_k^{-1}R)^{-\b_1}=r_k^{-1}$ and take $p=\frac{2}{\b_1}-1$ in \eqref{peqn}. Then by Lemma \ref{high2}
\[\a^6|\text{Br}_\a^{K}\cap\Lambda_k|\lesssim_\e K^6 R^{\e} \frac{1}{R^{\b(\frac{2}{\b_1}-1)}} C_\e R^\e r_k^{-1}R\sum_{\g_k}\|f_{\g_k}^{k+1}\|_{L^{4+\frac{4}{\b_1}}(\R^3)}^{4+\frac{4}{\b_1}}. \]
Majorize each $L^{4+\frac{4}{\b_1}}$ norm by a combination of $L^\infty $ and $L^6$ norms to get
\[\a^6|\text{Br}_\a^{K}\cap\Lambda_k|\lesssim_\e K^6 R^{2\e}\frac{1}{R^{\b(\frac{2}{\b_1}-1)}}  r_k^{-1}R\sum_{\g_k}\max_{\g_k}\|f_{\g_k}^{k+1}\|_\infty^{\frac{4}{\b_1}-2}\|f_{\g_k}^{k+1}\|_{L^{6}(\R^3)}^{6}. \]

Repeat the ``unwinding the pruning" argument from Subcase 1a to obtain 
\[\|f_{\g_k}^{k+1}\|_{L^{6}(\R^3)}^{6} \lesssim B_{\e^5}R^{\e^4}\big(\sum_{\g\subset\g_k}\|f_{\g}\|_{L^{6}(\R^3)}^{2}\big)^3\lesssim B_{\e^5}R^{\e^4}(r_k^{-1}R^\b)^2\sum_{\g\subset\g_k}\|f_{\g}\|_{L^{2}(\R^3)}^{2}  \]
where we used Cauchy-Schwarz and the assumption $\|f_\g\|_\infty\lesssim 1$ in the final inequality. Note that we have the additional constant $B_{\e^5}^{\e^{-1}}R^{\e^4}$ due the accumulation of $\le \e^{-1}$ many factors of the upper bound $B_{\e^5}R^{\e^5}$ for $L^6$ decoupling of the parabola with small parameter $\e^5$. In summary, \[\a^6|\text{Br}_\a^{K}\cap\Lambda_k|\lesssim_\e K^6 R^{3\e} \frac{1}{R^{\b(\frac{2}{\b_1}-1)}}  r_k^{-1}R\sum_{\g_k}\max_{\g_k}\|f_{\g_k}^{k+1}\|_\infty^{\frac{4}{\b_1}-2}(r_k^{-1}R^\b)^2\sum_{\g\subset\g_k}\|f_{\g}\|_{L^{2}(\R^3)}^{2}. \]

It suffices to check that $\frac{1}{R^{\b(\frac{2}{\b_1}-1)}}r_k^{-1}R\max_{\g_k}\|f_{\g_k}^{k+1}\|_\infty^{\frac{4}{\b_1}-1}(r_k^{-1}R^\b)^2\lessapprox R^{2\b+1}\a^{-\frac{2}{\b}}$, which simplifies to $R^{\b(1-\frac{2}{\b_1})}r_k^{-3}\max_{\g_k}\|f_{\g_k}\|_\infty^{\frac{4}{\b_1}-2}\lessapprox \a^{-\frac{2}{\b}}$.
Using $\|f_{\g_k}^{k+1}\|_\infty\le K^3A_\e^{(M-k)}\frac{R^\b}{\a}$, it further suffices to verify the inequality $r_k^{-3}R^{\b(\frac{2}{\b_1}-1)}\lessapprox \a^{\frac{4}{\b_1}-2-\frac{2}{\b}}$.

Suppose that the exponent $\frac{4}{\b_1}-2-\frac{2}{\b}\ge 0$. Use $r_k^{-1}\le R^{-1/3}$ and $R^\b\le \a^2$ to verify   
\begin{align*}
(R^\b)^{\frac{2}{\b_1}-1-\frac{1}{\b}}\le (\a^2)^{\frac{2}{\b_1}-1-\frac{1}{\b}}\qquad &\implies\qquad (R^{-1})R^{\b(\frac{2}{\b_1}-1)}\le \a^{\frac{4}{\b_1}-2-\frac{2}{\b}}. 
\end{align*}
Now suppose that the exponent $\frac{4}{\b_1}-2-\frac{2}{\b}<0$. Using Cauchy-Schwarz, the locally constant property, and the definition of $\Lambda_k$, for $x\in U_\a\cap\Lambda_k$, we have $\a^2\lesssim \#\g_{k+1}\sum_{\g_k}|f_{\g_{k+1}}^{k+2}|^2\lesssim R^\e r_kg_{k+1}(x)\lesssim R^\e r_kA_\e^{(M-k-1)}R^\b$. Also use $r_k^{1/\b_1}=r_k^{-1}R$ to verify
\begin{align*}
R^{-1}\le r_k^{-\frac{1}{\b}}\qquad&\implies\qquad r_k^{-3}R\le (r_k^{-1}R)^2r_k^{-1-\frac{1}{\b}}\\
&\implies \qquad r_k^{-3}R\le r_k^{\frac{2}{\b_1}-1-\frac{1}{\b}}\\
&\implies \qquad r_k^{-3}R(R^\e A_\e^{(M-k-1)} R^\b)^{\frac{2}{\b_1}-1-\frac{1}{\b}}\le (\a^2)^{\frac{2}{\b_1}-1-\frac{1}{\b}}\\
&\implies \qquad r_k^{-3}R^{\b(\frac{2}{\b_1}-1)}\le (R^\e A_\e^{(M-k-1)} R^\b)^{8}\a^{\frac{4}{\b_1}-2-\frac{2}{\b}}, 
\end{align*}
as desired.

\vspace{2mm}
\noindent\underline{Case 2: bounding $|\text{Br}_\a^{K}\cap\Omega_m|$.} Repeat the reasoning at the beginning of Case 1. By Lemma \ref{ftofk}, 
\[ |\text{Br}_\a^K\cap\Omega_m|\le|\{x\in  U_\a\cap\Omega_m:\a\lesssim K|F_{\tau_1^1}^{m+1}(x)F_{\tau_1^2}^{m+1}(x)F_{\tau_1^3}^{m+1}(x)|^{\frac{1}{3}},\quad\max_{\tau_1^i}|f_{\tau_1^i}(x)|\le \a\}.\]
Let $\{B_{R_m^{\frac{1}{3}}}\}$ be a finitely overlapping cover of $\text{Br}_\a^{K}\cap\Omega_m$ by $R_m^{\frac{1}{3}}$-balls. Then by Proposition \ref{trirestprop}, for $R$ large enough depending on $\e$, 
\begin{align*}
    \int_{B_{R_m^{\frac{1}{3}}}} |F_{\tau_1^1}^{m+1}F_{\tau_1^2}^{m+1}F_{\tau_1^3}^{m+1}|^2&\lesssim_\e R^{\e} |B_{R_m^{\frac{1}{3}}}|^{-2}\Big(\int|F_{\tau_1^1}^{m+1}|^2W_{B_{R_m^{\frac{1}{3}}}}\Big)\Big(\int|F_{\tau_1^2}^{m+1}|^2W_{B_{R_m^{\frac{1}{3}}}}\Big)\Big(\int|F_{\tau_1^3}^{m+1}|^2W_{B_{R_m^{\frac{1}{3}}}}\Big).
\end{align*}
The integrals on the right hand side are bounded by 
\[ C_\e \int\sum_{\tau_m}|F_{\tau_m}^{m+1}|^2W_{B_{R_m^{\frac{1}{3}}}} \]
using local $L^2$-orthogonality (Lemma \ref{L2orth}). 
If $x\in \text{Br}_\a^{K}\cap\Omega_m\cap B_{R_m^{\frac{1}{3}}}$, then the above integral is bounded by 
\[ C_\e \int \sum_{\tau_m}|F_{\tau_m}^{m+1}|^2*\w_{\tau_m}W_{B_{R_m^{\frac{1}{3}}}}\lesssim C_\e  \sum_{\tau_m}|F_{\tau_m}^{m+1}|^2*\w_{\tau_m}(x)=C_\e G_m(x) \]
by the locally constant property. Recall that since $x\in\Omega_m$, we have the lower bound $A_\e^{M+N-m}R^\b\le G_m(x)$. Also, by Corollary \ref{highdom}, $G_m(x)\le 2|G_m^h(x)|$. Combining the information so far yields
\[ \a^6|\text{Br}_\a^{K}\cap\Omega_m\cap B_{R_m^{\frac{1}{3}}}|\lesssim_\e K^6R^\e \frac{1}{(A_\e^{M+N-m}R^\b)^3}|B_{R_m^{\frac{1}{3}}}||G_m^h(x)|^6. \]
Use the locally constant property for $G_m^h$ and sum over all $B_{R_m^{\frac{1}{3}}}$ to get
\begin{equation*}\label{Peqn} \a^6|\text{Br}_\a^{K}\cap\Omega_m|\lesssim_\e K^6R^\e  \frac{1}{R^{3\b}}\int_{\R^3}|G_m^h|^6 .\end{equation*}
Note that we dropped the unnecessary factors of $A_\e^{M+N-m}\ge 1$ and that we are done using the properties of the set $\text{Br}_\a^{R_m^{\frac{1}{3}}}(\tau,\tau',\tau'')$, which is why we now integrate over all of $\R^3$ on the right hand side. 

By Lemma \ref{high3}, 
\[ \int_{\R^3}|G_m^h|^6\lesssim_\e R^\e \big(\sum_{\tau_m}\|F_{\tau_m}^{m+1}\|_{L^{12}(\R^3)}^4\big)^3. \]
Use Cauchy-Schwarz and then \eqref{item2} (with $F_{\tau_{m+1}}^{m+1}$) of Lemma \ref{pruneprop} to bound the $L^{12}$ norm by a combination of $L^\infty$ and $L^6$ norms:
\begin{align*}
\big(\sum_{\tau_{m}}\|F_{\tau_{m}}^{m+1}\|_{L^{12}(\R^3)}^4\big)^3\le R^\e K^6\big( K^3 A_\e^{M+N-m}\frac{R^\b}{\a}\big)^6\big(\sum_{\tau_{m+1}}\|F_{\tau_{m+1}}^{m+1}\|_{L^6(\R^3)}^2\big)^3.  
\end{align*}
Next, we use cylindrical $L^6$ decoupling over the parabola to unwind the pruning process. For each $\tau_{m+1}$, we have
\begin{align*}
(\text{\eqref{item1} of Lemma \ref{pruneprop}})\qquad\qquad \|F_{\tau_{m+1}}^{m+1}\|_{L^6(\R^3)}^6&\le \|F_{\tau_{m+1}}^{m+2}\|_{L^6(\R^3)}^6 \\
(\text{$L^6$ cyl. dec. for $\P^1$})\qquad\quad \,\,\qquad \qquad\qquad &\le  B_{\e^5}R^{\e^5}\big(\sum_{\tau_{m+2}\subset\tau_{m+1}}\|f_{\tau_{m+2}}^{m+2}\|_{L^6(\R^3)}^2\big)^3  \\
(\text{iterate previous two inequalities})\qquad\qquad\qquad &\le\cdots\le (B_{\e^5}R^{\e^5})^N \big(\sum_{\tau_{N}\subset\tau_{m+1}}\|f_{\tau_{N}}^{N+1}\|_{L^6(\R^3)}^2\big)^3 . 
\end{align*}
Note that $\{\tau_N\}$ are canonical blocks of the moment curve. Our goal is to have an expression involving the small caps $\g$. We defined the $\g$ so that they lie in the cylindrical region over canonical $R^{-\b}\times R^{-2\b}$ blocks of $\P^1$. Therefore, we may continue unwinding the pruning process using Theorem \ref{cyldec}, ultimately obtaining 
\[ \big(\sum_{\tau_{m+1}}\|F_{\tau_{m+1}}^{m+1}\|_{L^6(\R^3)}^2\big)^3\le(B_{\e^5}R^{\e^5})^{M+N} \big(\sum_{\g}\|f_{\g}\|_{L^6(\R^3)}^2\big)^3 . 
 \]
By Cauchy-Schwarz and using the assumption $\|f_\g\|_\infty\lesssim 1$, we have 
\[ \big(\sum_{\g}\|f_{\g}\|_{L^6(\R^3)}^2\big)^3\le \#\g^2\sum_\g\|f_\g\|_{L^6(\R^3)}^6\lesssim R^{2\b}\sum_\g\|f_\g\|_{L^2(\R^3)}^2. \]
The summary in this case is that 
\[ \a^6|\text{Br}_\a^{K}\cap\Omega_m|\lesssim_\e K^{30}   R^{3\e}A_\e^{10(M+N)} \frac{1}{R^{3\b}}\Big(\frac{R^\b}{\a}\Big)^6(R^{2\b})\sum_\g\|f_\g\|_{L^2(\R^3)}^2.  \]
It suffices to verify that $R^{5\b}\a^{-6}\le R^{2\b+1}\a^{-\frac{2}{\b}}$. This follows immediately from the relation $R^\b\le \a^2$. 

\noindent\underline{Case 3: bounding $|U_\a\cap L|$.} 
Begin by using Lemma \ref{ftofk} to bound
\[ \a^{6+\frac{2}{\b}}|\text{Br}_\a^K\cap L|\lesssim K^{12} \int_{U_\a\cap L}|f|^2|F_1|^{4+\frac{2}{\b}}. \]
Then use Cauchy-Schwarz and the locally constant property for $G_1$ :
\[ \int_{U_\a\cap L}|f|^2|F_1|^{4+\frac{2}{\b}}\lesssim_\e R^\e \int_{U_\a\cap L}|f|^2G_1^{2+\frac{1}{\b}}.\]
Using the definition the definition of $L$, we bound the factors of $G_1$ by
\[ \int_{U_\a\cap L}|f|^2 (A_\e^{M+N}R^\b)^{2+\frac{1}{\b}}. \]
Finally, use $L^2$ orthogonality to conclude
\[    \a^{6+\frac{2}{\b}}|\text{Br}_\a^K\cap L|\lesssim_\e K^{12}R^{2\e}A_\e^{10(M+N)}R^{2\b+1}\sum_\g\|f_\g\|_{L^2(\R^3)}^2. \]

\end{proof}

\subsection{Wave packet decomposition and pigeonholing \label{M3pigeon}}

To prove Theorem \ref{main}, it suffices to prove a local version presented in the next lemma. 

\begin{lemma}\label{loc} Let $\frac{1}{3}\le \b\le 1$ and $p\ge 2$. Then for any $R^{\max(2\b,1)}$-ball $B_{R^{\max(2\b,1)}}\subset\R^3$, suppose that 
\[     \|f\|_{L^p(B_{R^{\max(2\b,1)}})}^p\le C_\e R^\e(R^{\b(\frac{p}{2}-1)}+R^{\b(p-4)-1})\sum_\g\|f_\g\|_{L^p(\R^3)}^p \]
for any Schwartz function $f:\R^3\to\C$ with Fourier transform supported in $\mc{M}^3(R^\b,R)$.  Then Theorem \ref{main} is true. 
\end{lemma}
\begin{proof} Write 
\[ \|f\|_{L^p(\R^3)}^p\lesssim \sum_{B_{R^{\max(2\b,1)}}}\int_{B_{R^{\max(2\b,1)}}}|f|^p\]
where the sum is over a finitely overlapping cover of $\R^3$ by $R^{\max(2\b,1)}$-balls. Let $\phi_{B}$ be a weight function decaying by order $100$ away from $B_{R^{\max(2\b,1)}}$, satisfying $\phi_{B}\gtrsim 1$ on $B_{R^{\max(2\b,1)}}$, and with Fourier transform supported in an $R^{-2}$ neighborhood of the origin. The Fourier support of each $f_\g\phi_{B}$ is contained in a $2R^{-\b}\times 4R^{-2\b}\times 2^{\frac{1}{\b}}R^{-1}$ small cap. By the triangle inequality, there is a subset $\mc{S}$ of the small caps $\g$ so that for each $\g\in\mc{S}$, the Fourier support of $f_{\g}\phi_{B}$ is contained in a unique small cap and  
\[ \|f\|_{L^p(B_R)}^p\lesssim \|\sum_{\g\in\mc{S}}f_\g\phi_{B}\|_{L^p(B_{R^2})}^p. \]
Then by applying the hypothesized local version of small cap decoupling, 
\[ \|\sum_{\g\in\mc{S}}f_\g\phi_{B}\|_{L^p(B_{R^2})}^p\le C_\e R^\e(R^{\b(\frac{p}{2}-1)}+R^{\b(p-4)-1})\sum_{\g\in\mc{S}}\|f_\g\phi_{B}\|_{L^p(\R^3)}^p .\]
It remains to note that $\sum_{B_{R^2}}\int |f_\g|^p\phi_{B}^p\lesssim \int|f_\g|^p$. 
\end{proof}

It further suffices to prove a weak, level-set version of Theorem \ref{main}. 
\begin{lemma}\label{alph} Let $p\ge 2$. For each $B_{R^2}$ and Schwartz function $f:\R^3\to\C$ with Fourier transform supported in $\mc{M}^3(R^\b,R)$, there exists $\a>0$ such that
\[ \|f\|_{L^p(B_{R^2})}^p\lesssim_p (\log R)\a^p|\{x\in B_{R^2}:\a\le |f(x)|\}|+R^{-500p}\sum_\g\|f_\g\|_{L^p(\R^3)}^p. \]
\end{lemma}
\begin{proof} Split the integral as follows:
\[ \int_{B_{R^2}}|f|^p=\sum_{R^{-1000}\le \lambda\le 1}\int_{\{x\in B_{R^2}:\a\|f\|_{L^\infty(B_{R^2})}\le |f(x)|\le 2\a \|f\|_{L^\infty(B_{R^2})}\}}|f|^p+\int_{\{x\in B_{R^2}:|f(x)|\le R^{-1000}\|f\|_{L^\infty(B_{R^2})}\}}|f|^p\]
in which $\lambda$ varies over dyadic values in the range $[R^{-1000},1]$. If one of the $\lesssim \log R$ many terms in the first sum dominates, then we are done. Suppose instead that the second expression dominates:
\[ \int_{B_{R^2}}|f|^p\le 2\int_{\{x\in B_{R^2}:|f(x)|\le R^{-1000}\|f\|_{L^\infty(B_{R^2})}\}}|f|^p\lesssim R^3R^{-1000p}\|f\|_{L^\infty(B_{R^2})}^p. \]
Then by H\"{o}lder's inequality, we have
\[ \int_{B_{R^2}}|f|^p\lesssim R^3R^{-1000p+(p-1)}\sum_\g\|f_\g\|_{L^\infty(B_{R^2})}^p. \]
Finally, by the locally constant property and H\"{o}lder's inequality, 
\[\|f_\g\|_\infty(B_{R^2})^p\lesssim \||f_\g|*\w_{\g^*}\|_{L^\infty(B_{R^2})}^p\lesssim_p \||f_\g|^p*\w_{\g^*}\|_{L^\infty(B_{R^2})}\lesssim \int_{\R^3}|f_\g|^p. \]

\end{proof}

Use the notation 
\[ U_\a=\{x\in B_{R^2}:\a\le |f(x)|\}. \]
We will show that to estimate the size of $U_\a$, it suffices to replace $f$ with a version whose wave packets have been pigeonholed. Write 
\begin{align}\label{sum} f=\sum_\g\sum_{T\in\T_\g}\s_Tf_\g \end{align}
where for each $\g$, $\{\s_T\}_{T\in\T_\g}$ is the partition of unity from a partition of unity from \textsection\ref{prusec}. If $\a\le C_\e (\log R)R^{-500}\max_\g\|f_\g\|_{\infty}$, then by an analogous argument as dealing with the low integral over $\{x:|f(x)|\le R^{-1000}\|f\|_\infty\}$ in the proof of Lemma \ref{alph}, bounding $\a^p|U_\a|$ by the right hand side of the small cap decoupling theorem is trivial. Let $\phi_{B}$ be the weight function from Lemma \ref{loc}. 

\begin{prop}[Wave packet decomposition] \label{wpd} Let $\a>C_\e (\log R) R^{-100}\max_\g\|f_\g\|_{L^\infty(\R^3)}$. There exist subsets $\mc{S}\subset\{\g\}$ and  $\tilde{\T}_\g\subset\T_\g$, as well as a constant $A>0$ with the following properties:
\begin{align} 
|U_\a|\lesssim (\log R)|\{x\in U_\a:\,\,\a&\lesssim |\sum_{\g\in\mc{S}}\sum_{T\in\tilde{\T}_\g}\s_T(x)\phi_{B}(x)f_\g (x)|\,\,\}|, \\
\|\sum_{T\in\tilde{\T}_\g}\s_T\phi_{B}f_\g\|_{L^\infty(\R^3)}&\sim A\qquad\text{for all}\quad  \g\in\mc{S}, \\
\text{and } \quad \#\tilde{\T}_\g A^p R^{\b+2\b+1}\lesssim\|\sum_{T\in\tilde{\T}_\g}\s_T&\phi_{B}f_\g\|_{L^p(\R^3)}\lesssim R^{3p\e }\#\tilde{\T}_\g A^p R^{\b+2\b+1}\quad\text{ for all}\quad \g\in{\mc{S}}. \label{Lpprop}
\end{align}
\end{prop}

\begin{proof} 
Split the sum (\ref{sum}) into 
\begin{equation}\label{step1} \phi_{B}f=\sum_\g\sum_{T\in\T_\g^c}\s_T\phi_{B}f_\g+\sum_\g\sum_{T\in\T_\g^f}\s_T\phi_{B}f_\g\end{equation} 
where the close set is
\[ \T_\g^c:=\{T\in\T_\g:T\cap R^{10} B_{R^2}\not=\emptyset\}\]
and the far set is 
\[ \T_\g^f:=\{T\in\T_\g:T\cap R^{10}B_{R^2}=\emptyset\} . \]
Using decay properties of the partition of unity, for each $x\in B_{R^2}$, 
\[ |\sum_\g\sum_{T\in\T_\g^f}\s_T(x)\phi_{B}(x)f_\g(x)|\lesssim R^{-1000}\max_\g\|\phi_{B}f_\g\|_{L^\infty(B_{R^2})}. \] 
Therefore, using the assumption that $\a$ is at least $R^{-100}\max_\g\|f_\g\|_{L^\infty(B_{R^2})}$, 
\[ |U_\a|\le 2|\{x\in U_\a:\,\,\a\le 2 |\sum_\g\sum_{T\in\tilde{\T}_\g^c}\s_T(x)\phi_{B}(x)f_\g(x)|\,\,\}|. \]
The close set  has cardinality $|\T_{\g}^c|\leq R^{33}$.  Let
 \begin{equation}\label{eq: defM}
	M=\max_\g\max_{T\in\T_\g^c}\|\s_T\phi_{B}f_\g\|_{L^\infty(\R^3)}.
	\end{equation} 
Split the remaining wave packets into 
\begin{equation} \label{step2}
    \sum_\g\sum_{T\in\T_\g^c}\s_T\phi_{B}f_\g=\sum_\g\sum_{R^{-10^3}\le \lambda\le 1}\sum_{T\in\T_{\g,\lambda}^c}\s_T\phi_{B}f_\g+\sum_\g\sum_{T\in\T_{\g,s}^c}\s_T\phi_{B}f_\g
\end{equation}
where $\lambda$ is a dyadic number in the range $[R^{-10^3},1]$,  
\[ \T_{\g,\lambda}^c:=\{T\in\T_\g^c:\|\s_T\phi_{B}f_\g\|_{L^\infty(\R^3)}\sim \lambda M \},\]
and
\[ \T_{\g,s}^c:= \{T\in\T_\g^c:\|\s_T\phi_{B}f_\g\|_{L^\infty(\R^3)}\le R^{-1000}M \} . \]
Again using the lower bound for $\a$, the small wave packets cannot dominate and we have 
\[ |U_\a|\le 4|\{x\in U_\a:\,\,\a\le 4|\sum_\g \sum_{R^{-10^3}\le \lambda\le 1}\sum_{T\in\T_{\g,\lambda}^c}\s_T(x)\phi_{B}(x)f_\g(x)|\,\,\}|.\]
By dyadic pigeonholing, for some $\lambda\in [R^{-1000},1]$, 
\[ |U_\a|\lesssim (\log R)|\{x\in U_\a:\,\,\a\lesssim (\log R)|\sum_\g \sum_{T\in\T_{\g,\lambda}^c}\s_T(x)\phi_{B}(x)f_\g(x)|\,\,\}|. \]
Finally, we analyze the $L^p$ norm for each $p\ge 2$ and each $\g$. Note that we have the pointwise inequality 
\begin{align*} 
|\sum_{T\in\T_{\g,\lambda}^c}\s_T(x)\phi_{B}(x)f_\g(x)|&= |\sum_{\substack{T\in\T_{\g,\lambda}^c\\ x\in R^\e T}}\s_T(x)\phi_{B}(x)f_\g(x)|+|\sum_{\substack{T\in\T_{\g,\lambda}^c\\ x\not\in R^\e T}}\s_T(x)\phi_{B}(x)f_\g(x)| \\
&\le  |\sum_{\substack{T\in\T_{\g,\lambda}^c\\ x\in R^\e T}}\s_T(x)\phi_{B}(x)f_\g(x)|+C_\e R^{-1000}|\phi_{B}(x)f_\g(x)| .
\end{align*}
Let $\mc{S}'$ be the subset of $\{\g\}$ for which 
\[ \|\sum_{T\in\T_{\g,\lambda}^c}\s_T\phi_{B} f_\g\|_{L^\infty(\R^3)}\ge C_\e R^{-500}\max_\g\|\phi_{B}f_\g\|_{L^\infty(\R^3)}.  \]
Using the lower bound for $\a$, we then have 
\[ |U_\a|\lesssim (\log R)|\{x\in U_\a:\,\,\a\lesssim (\log R)|\sum_{\g\in\mc{S}'} \sum_{T\in\T_{\g,\lambda}^c}\s_T(x)\phi_{B}(x)f_\g(x)|\,\,\}|.\]
It follows from the pointwise inequality above that for each $\g\in\mc{S}'$, 
\[ \lambda M\lesssim \|\sum_{T\in\T_{\g,\lambda}^c}\s_T \phi_{B}f_\g\|_{L^\infty(\R^3)}\lesssim R^{3\e} \lambda M.\]
Perform one more dyadic pigeonholing step to obtain a dyadic $\mu\in[1,R^\e]$ for which 
\[ |U_\a|\lesssim (\log R)^2|\{x\in U_\a:\a\lesssim (\log R)^2|\sum_{\g\in\mc{S}}\sum_{T\in\T_{\g,\lambda}^c}\s_T(x)\phi_{B}(x)f_\g(x)|\}| \]
where $\mc{S}$ is the set of $\g$ satisfying $\|\sum_{T\in\T_{\g,\lambda}^c}\s_T\phi_{B} f_\g\|_{L^\infty(\R^3)}\sim\mu M$. 

It remains to check the property about the $L^p$ norms. For each $\g\in\mc{S}$, using the locally constant property, we have
\begin{align*}
\#\T_{\g,\lambda}^cR^{\b+2\b+1}(\mu M)^p&\lesssim \sum_{T\in\T_{\g,\lambda}^c}\int|\s_T\phi_{B}f_\g|^p \lesssim  \int|\sum_{T\in\T_{\g,\lambda}^c}\s_T\phi_{B}f_\g|^p \\
    &\lesssim \int|\sum_{\substack{T\in\T_{\g,\lambda}^c\\ x\in R^\e T}}\s_T(x)f_\g(x)|^pdx+C_\e R^{-1000p}\|\phi_{B}f_\g\|_{L^p(\R^3)}^p \\
    &\lesssim R^{3p\e}\#\T_{\g,\lambda}^c R^{\b+2\b+1}(\mu M)^p+C_\e R^{-1000p}\|\phi_{B}f_\g\|_{L^p(\R^3)}^p.
\end{align*}
By construction, we have $M\ge C_\e R^{-501}\max_\g\|f_\g\|_{L^\infty(\R^3)}$. It follows that 
\[ C_\e R^{-1000p}\|\phi_{B}f_\g\|_{L^p(\R^3)}^p\lesssim R^{-100}\#\T_{\g,\lambda}^c R^{\b+2\b+1}(\mu M)^p \]
which concludes the proof.

\end{proof}

\subsection{Trilinear reduction \label{broadnarrow}}

We will present a broad/narrow analysis to show that Proposition \ref{mainprop} implies the following level set version of Theorem \ref{main}, for the critical $p=6+\frac{2}{\b}$. 

\begin{theorem}\label{alphamain} For any $R\ge 2$, $\frac{1}{3}\le\b\le1$, and $\a>0$, 
\[ \a^{6+\frac{2}{\b}}|U_\a|\lesssim_\e R^{O(\e)} R^{2\b+1} \sum_\g\|f_\g\|_2^2\]
for any Schwartz function $f:\R^3\to\C$ with Fourier transform supported in $\mc{M}^3(R^\b,R)$ and satisfying $\|f_\g\|_\infty\le 2$ for all $\g$. 
\end{theorem}

\begin{proof}[Proposition \ref{mainprop} implies Theorem \ref{alphamain}]

We present an algorithm incorporating a broad-narrow argument. For each $k$, $1\le k\le N$, recall that $\{\tau_k\}$ is a collection of canonical $\sim R_k^{-\frac{1}{3}}\times R_k^{-\frac{2}{3}}\times R_k^{-1}$ moment curve blocks. Write $\ell(\tau)=r^{-1}$ to denote that $\tau$ is a canonical $r^{-1}\times r^{-2}\times r^{-3}$ moment curve block. 

Step 1 of the algorithm is as follows. Let $E_\e$ be a constant we choose to be larger than $10 \underline{C}_\e$, where $\underline{C}_\e$ is from Lemma \ref{pruneprop}. We have the broad/narrow inequality
\begin{align}\label{brnar}
    |f(x)|&\le 4E_\e\max_{\tau_1}|f_{\tau_1}(x)|+R^{2\e} \max_{\substack{d(\tau_1^i,\tau_1^j)\ge E_\e R_1^{-\frac{1}{3}}}}|f_{\tau_1^1}(x)f_{\tau_1^2}(x)f_{\tau_1^3}(x)|^{\frac{1}{3}} 
\end{align}
where the second term is the maximum over 3-tuples of $\tau_1$ which are pairwise $\ge E_\e R_1^{-\frac{1}{3}}$-separated. 
Indeed, suppose that the set $\{\tau_1:|f_{\tau_1}(x)|\ge R_1^{-\frac{1}{3}}\max_{\tau_1'}|f_{\tau_1'}(x)|\}$ has at least $3E_\e$ elements. Then we can find three $\tau_1^1,\tau_1^2,\tau_1^3$ which are pairwise $\ge E_\e R_1^{-\frac{1}{3}}$-separated and satisfy $|f(x)|\le R^{2\e}|f_{\tau_1^1}(x)f_{\tau_1^2}(x)f_{\tau_1^3}(x)|^{\frac{1}{3}}$. If there are fewer than $3E_\e$ elements, then $|f(x)|\le 3E_\e \max_{\tau_1'}|f_{\tau_1'}(x)|+\max_{\tau_1'}|f_{\tau_1'}(x)|$.

Suppose that
\[ |U_\a|\le 2|\{x\in U_\a:\max_{\tau_1}|f_{\tau_1}(x)|\le\a\}|. \]
If this does not hold, then proceed to Step 2 of the algorithm. Further suppose that there are blocks $\tau_1^i$ which satisfy $d(\tau_1^i,\tau_1^j)\ge E_\e R_1^{-\frac{1}{3}}$ and  
\begin{align}\label{step1alg}
    |U_\a|\lesssim R^{3\e} |\{x\in U_\a:\a\le 2R^{2\e}|f_{\tau_1^1}(x)f_{\tau_1^2}(x)f_{\tau_1^3}(x)|^{\frac{1}{3}},\quad\max_{\tau_1}|f_{\tau_1}(x)|\le \a\}|. 
\end{align}
If \eqref{step1alg} does not hold, then proceed to Step 2 of the algorithm. 
Assuming \eqref{step1alg}, apply Proposition \ref{mainprop} to get the inequality
\[ \a^{6+\frac{2}{\b}}|U_\a|\lesssim_\e R^{O(\e)} R^{2\b+1}\sum_{{\g}}\|{f}_\g\|_2^2 ,\]
which terminates the algorithm. 

Next, we describe step $k$ of the algorithm for $k\ge 2$ and $ R_{k-1}^{\frac{2}{3}} \le R^{1-\b}$. The input for step $k$ is 
\begin{align}\label{stepkalg}
    |U_\a|\lesssim_\e(\log R)^{k-1}|\{x\in U_\a:\a\lesssim (\log R)^{k-1}\max_{{\tau}_{k-1}}|f_{{\tau}_{k-1}}(x)|\}| .
\end{align}
For each $\tau_{k-1}$, we have the broad-narrow inequality 
\begin{align*}
|f_{{\tau}_{k-1}}(x)|&\le 2E_\e\max_{\substack{{\tau}_k\subset{\tau}_{k-1}}}|f_{{\tau}_k}(x)|+R^{2\e}\max_{\substack{\tau_k^i\subset{\tau}_{k-1} \\d(\tau_k^i,\tau_k^j)\ge E_\e R_k^{-\frac{1}{3}}}}|f_{\tau_k^1}(x)f_{\tau_k^2}(x)f_{\tau_k^3}(x)|^{\frac{1}{3}}. 
\end{align*} 
Either proceed to Step $k+1$ or assume that 
\[ |U_\a|\lesssim (\log R)^{k-1} |\{x\in U_\a:\a\lesssim (\log R)^{k-1}\max_{\tau_{k-1}}|{f}_{\tau_{k-1}}(x)|,\quad \max_{\tau_k}|f_{\tau_k}(x)|\le \a\}| . \]
Again, either proceed to Step $k+1$ or assume further that 
there are $\tau_k^i\subset\tau_{k-1}$ which are pairwise $\ge E_\e R_k^{-\frac{1}{3}}$-separated  and satisfy 
\[ |U_\a|\le (\log R)^kR^{3\e}\sum_{\tau_{k-1}}|\{x\in U_\a:\a\lesssim (\log R)^{k-1}R^{\e}|f_{{\tau}_{k}^1}(x)f_{\tau_k^2}(x)f_{\tau_k^3}(x)|^{\frac{1}{3}},\quad \max_{\tau_k}|f_{\tau_k}(x)|\le \a\}|. \]
By rescaling for the moment curve, there exists a linear transformation $T$ so that $|f_{\tau_k^i}\circ T|=|g_{\underline{\tau}_k^i}|$ where the $\underline{\tau}_k^i$ are pairwise $\gtrsim E_\e R_1^{-\frac{1}{3}}$-separated blocks and $g$ is Fourier supported in the anisotropic neighborhood $\mc{M}^3(R_{k-1}^{-\frac{1}{3}}R^\b,R_{k-1}^{-1}R)$. Indeed, suppose that ${\tau}_{k-1}$ is the $l$th piece
\[ {\tau}_{k-1}=\{(\xi_1,\xi_2,\xi_3): lR_{k-1}^{-\frac{1}{3}}\le \xi_1<(l+1)R_{k-1}^{-\frac{1}{3}},\,|\xi_2-\xi_1^2|\le R_{k-1}^{-\frac{2}{3}},\,|\xi_3-3\xi_1\xi_2+2\xi_1^3|\le R_{k-1}^{-1} \}  .\]
Since the Fourier support of $f$ is in $\mc{M}^3(R^\b,R)$ by hypothesis, the Fourier support of $f_{{\tau}_{k-1}}$ is in ${\tau}_{k-1}\cap\mc{M}^3(R^\b,R)$. Define the affine transformation $L(\xi_1,\xi_2,\xi_3)$ by 
\begin{align*}
\xi_1&\mapsto \,\,R_{k-1}^{\frac{1}{3}}(\xi_1-lR_{k-1}^{-\frac{1}{3}})\\
\xi_2&\mapsto \,\,R_{k-1}^{\frac{2}{3}}(\xi_2-l^2R_{k-1}^{-\frac{2}{3}})-2lR_{k-1}^{\frac{1}{3}}(\xi_1-lR_{k-1}^{-\frac{1}{3}})\\
\xi_3&\mapsto \,\,R_{k-1}(\xi_3-l^3R_{k-1}^{-1})-3lR_{k-1}^{\frac{2}{3}}(\xi_2-l^2R_{k-1}^{-\frac{2}{3}})+3l^2R_{k-1}^{\frac{1}{3}}(\xi_1-lR_{k-1}^{-\frac{1}{3}}) .
\end{align*}
This affine map satisfies $L({\tau}_{k-1}\cap\mc{M}^3(R^\b,R))=\mc{M}^3(R_{k-1}^{-\frac{1}{3}}R^\b,R_{k-1}^{-1}R)$. If we write $L^{-1}(\xi_1,\xi_2,\xi_3)=A(\xi_1,\xi_2,\xi_3)+b$ where $A$ is a linear map, then the rescaling map $T$ above is equal to $(A^{-1})^*$. In this step, we have assumed that $R_{k-1}R^{-1}\le R_{k-1}^{\frac{1}{3}}R^{-\b}$. One may then verify that $L(\g)=\underline{\g}$ are $\sim R_{k-1}^{\frac{1}{3}}R^{-\b}\times R_{k-1}^{\frac{2}{3}}R^{-2\b}\times R_{k-1}R^{-1}$ small caps partitioning $\mc{M}^3(R_{k-1}^{-\frac{1}{3}}R^\b,R_{k-1}^{-1}R)$. Apply Proposition \ref{mainprop} to the rescaled functions to obtain the inequality
\begin{align*}
    \a^{6+\frac{2}{\b'}}|\{x\in U_\a:\a\lesssim (\log R)^{k-1}R^{\e}|{g}_{\underline{\tau}_k^1}(x){g}_{\underline{\tau}_k^2}(x){g}_{\underline{\tau}_k^3}(x)|^{\frac{1}{3}},&\quad \max_{\underline{\tau}_k\subset\underline{\tau}_{k-1}}|{g}_{\underline{\tau}_k}(x)|\le \a\}| \\
    &\lesssim_\e R^{3\e+10\e} (R_{k-1}^{-1}R)^{2\b'+1}\sum_{\underline{\g}}\|g_{\underline{\g}}\|_2^2.
\end{align*} 
where $\b'\in[\frac{1}{3},1]$ is defined by $(R_{k-1}R^{-1})^{\b'}=R_{k-1}^{\frac{1}{3}}R^{-\b}$. By undoing the rescaling change of variables and summing over $\tau_{k-1}$, this implies 
\[ \a^{6+\frac{2}{\b'}}|U_\a|\lesssim_\e R^{13\e} (R_{k-1}^{-1}R)^{2\b'+1}\sum_{\g}\|f_{\g}\|_2^2 .\]
It suffices to verify that $(R_{k-1}^{-1}R)^{2\b'+1}\lessapprox \frac{R^{2\b+1}}{\a^{\frac{2}{\b}-\frac{2}{\b'}}}$. Use the upper bound $\a\lessapprox R_{k-1}^{-\frac{1}{3}}R^\b$ from the step we are considering so that it suffices to verify $(R_{k-1}^{-1}R)^{2\b'+1}(R_{k-1}^{-\frac{1}{3}}R^\b)^{\frac{2}{\b}-\frac{2}{\b'}}\lessapprox R^{2\b+1} $, which simplifies to $R_{k-1}^{-2\b'-1-\frac{2}{3\b}+\frac{2}{3\b'}}\lessapprox R^{2\b-2\b'-2+\frac{2\b}{\b'}}$. Using the definition of $\b'$, this further simplifies to $R_{k-1}^{-2\b'-1-\frac{2}{3\b}+\frac{2}{3\b'}}\lessapprox R_{k-1}^{(-\b'+\frac{1}{3})(2+\frac{2}{\b'})}$, which is true since $\b\le 2$. In this case, the algorithm terminates. 


Next, we describe step $k$ with $k\ge 2$ and $R_{k-1}^{\frac{2}{3}}\ge R^{1-\b}$. The input for step $k$ is 
\begin{align}\label{stepkbalg}
    |U_\a|\le (\log R)^{k-1}|\{x\in U_\a:\a\lesssim (\log R)^{k-1}\max_{\tau_{k-1}}|f_{{\tau}_{k-1}}(x)|\}|. 
\end{align}
Let $\{{\z}\}$ be a partition of $\mc{M}^3(R^\b,R)$ into $\sim R_{k-1}^{\frac{2}{3}}R^{-1}\times R_{k-1}^{\frac{4}{3}}R^{-2}\times R^{-1}$ small caps. By Lemma \ref{wpd}, we may assume that there are versions $\tilde{f}_{\tau_{k-1}}$ of the $f_{\tau_{k-1}}$ whose wave packets corresponding to $\z$ have been localized and pigeonholed and which satisfy
\[ |U_\a|\lesssim (\log R)^k |\{x\in U_\a:\a\lesssim (\log R)^{k} \max_{\tau_{k-1}}|\tilde{f}_{\tau_{k-1}}(x)|\}|.\]
As in the previous case, either we proceed to Step $k+1$ or we have
\[ |U_\a|\le (\log R)^kR^{3\e}\sum_{\tau_{k-1}}|\{x\in U_\a:\a\lesssim (\log R)^{k}|\tilde{f}_{{\tau}_{k}^1}(x)\tilde{f}_{\tau_k^2}(x)\tilde{f}_{\tau_k^3}(x)|^{\frac{1}{3}},\quad \max_{\tau_k\subset\tau_{k-1}}|\tilde{f}_{\tau_k}(x)|\le \a\}|. \]
By the same rescaling argument as above, let $T$ be the linear transformation so that $|\tilde{f}_{\tau_k^i}\circ T|=|g_{\underline{\tau}_k^i}|$ and the $\underline{\tau}_k^i$ are pairwise $\gtrsim E_\e R_1^{-\frac{1}{3}}$-separated blocks and $g$ is Fourier supported in the anisotropic neighborhood $\mc{M}^3(R_{k-1}^{-\frac{1}{3}}R^\b,R_{k-1}^{-1}R)$. Note that each $|\tilde{f}_{\z}\circ T|=|g_{\underline{\z}}|$ where $\underline{\z}$ is an $R_{k-1}R^{-1}\times R_{k-1}^2R^{-2}\times R_{k-1}R^{-1}$ small cap. Apply Proposition \ref{mainprop} to the rescaled functions $(\max_{\underline{\z}}\|{g}_{\underline{\z}}\|_\infty)^{-1}(g_{\underline{\tau}_k^1}+g_{\underline{\tau}_k^2}+g_{\underline{\tau}_k^3})$ to obtain the inequality
\begin{align*}
    \a^{8}|\{x\in U_\a:\a\le (\log R)^{k}|{g}_{\underline{\tau}_k^1}(x){g}_{\underline{\tau}_k^2}(x){g}_{\underline{\tau}_k^3}(x)|^{\frac{1}{3}},&\quad \max_{\underline{\tau}_k}|{g}_{\underline{\tau}_k}(x)|\le \a\}| \\
    &\lesssim_\e R^{10\e} (R_{k-1}^{-1}R)^{2(1)+1}\max_{\underline{\z}}\|g_{\underline{\z}}\|_\infty^6\sum_{\underline{\z}}\|g_{\underline{\z}}\|_2^2.
\end{align*} 
By undoing the rescaling change of variables and summing over $\tau_{k-1}$, this implies 
\[ \a^{8}|U_\a|\lesssim_\e R^{10\e} (R_{k-1}^{-1}R)^{3}(\max_{\z}\|\tilde{f}_\z\|_\infty)^6\sum_{\z}\|\tilde{f}_{\z}\|_2^2 .\]
By properties of the pigeonholing lemma, for each $\z$, $(\max_{\z}\|\tilde{f}_\z\|_\infty)^6\|\tilde{f}_{\z}\|_2^2\lesssim_\e R^{3\e}(R_{k-1}^{\frac{2}{3}}R^{-1}R^\b)^2\|f_\z\|_6^6$. By cylindrical $L^6$ decoupling (Theorem \ref{cyldec}), for each $\z$, $\|f_\z\|_6^6\lesssim_\e R^\e(\sum_{\g\subset\z}\|f_\g\|_6^2)^3\lesssim_\e R^\e(R_{k-1}^{\frac{2}{3}}R^{-1}R^\b)^2\sum_{\g\subset\z}\|f_\g\|_2^2$. The summary of step $k$ in this case is that
\[ \a^8|U_\a|\lesssim_\e R^{3\e+20\e} (R_{k-1}^{-1}R)^{3}(R_{k-1}^{\frac{2}{3}}R^{-1}R^\b)^{4}\sum_{\g}\|\tilde{f}_{\g}\|_2^2. \]
It remains to verify that $R_{k-1}^{-\frac{1}{3}}R^{4\b-1}\lessapprox\frac{R^{2\b+1}}{\a^{\frac{2}{\b}-2}}$. This is true since $R_{k-1}^{\frac{1}{3}}\ge 1$ and $\a\le R^\b$. The algorithm terminates in this case.

The final step, if the algorithm has not terminated yet, gives the case
\[ |U_\a|\lesssim (\log R)^{N}|\{x\in U_\a:\a\lesssim (\log R)^N\max_{{\tau_N}}|f_{{\tau}_N}(x)|\}|. \]
Write $\tau_N=\theta$ and use trivial inequalities: 
\begin{align*}
\a^{6+\frac{2}{\b}}|\{x\in U_\a:\a\lesssim (\log R)^{N}\max_{\theta}|f_{\theta}(x)|\}|&\lesssim_\e (\log R)^N\sum_\theta\int|f_\theta|^{6+\frac{2}{\b}} \\
&\lesssim_\e (\log R)^N\sum_\theta\max_\theta\|f_\theta\|_\infty^{4+\frac{2}{\b}}\int|f_\theta|^{2}\\
&\lesssim_\e (\log R)^N\sum_\theta\max_\theta(\#\g\subset\theta)^{4+\frac{2}{\b}}\int\sum_{\g\subset\theta}|f_\g|^{2}\\
&\lesssim_\e (\log R)^NR^{(\b-\frac{1}{2})(4+\frac{2}{\b})}\sum_\g\|f_\g\|_2^2
\end{align*}
where we used Lemma \ref{pruneprop} for the $L^\infty$ bound. Technically, our algorithm could give us a version of $f$ whose wave packets have been pigeonholed at a few scales. In that case, we incorporate an analogous process as ``unwinding the pruning" from the proof of Proposition \ref{mainprop} into the trivial argument above. Noting that $N\sim\e^{-1}$, and $(\log R)^N(\log R)^N\lesssim_\e R^\e$, we are done since $(\b-\frac{1}{2})(4+\frac{2}{\b})\le 2\b+1$, which is equivalent to $\b\le 1$. 

\end{proof}

\subsection{Proof that Theorem \ref{alphamain} implies Theorem \ref{main}}

We divide the work into two propositions. First, in Proposition \ref{critexp}, we show that Theorem \ref{alphamain} implies the critical exponent $p=6+\frac{2}{\b}$ version of Theorem \ref{main}. Then, we show that the general Theorem \ref{main} follows from the critical exponent case in Proposition \ref{genp}. 

\begin{proposition}\label{critexp} Theorem \ref{alphamain} implies Theorem \ref{main} for the critical exponent $p=6+\frac{2}{\b}$. 
\end{proposition}
\begin{proof} Fix $p=6+\frac{2}{\b}$. By Lemma \ref{loc}, it suffices to bound the $L^p$ norm of $f$ on a fixed ball $B_{R^{\max(2\b,1)}}$. By Lemma \ref{alph}, there is a constant $\a>0$ (which we may assume is $\ge C_\e (\log R)R^{-100}\max_\g\|f_\g\|_\infty$) so that it suffices to bound $\a^p|U_\a|$ for $U_\a=\{x\in B_{R^{\max(2\b,1)}}:\a\le|f(x)|\}$. Finally, by Proposition \ref{wpd}, we may replace $f$ by a pigeonholed and localized version $\tilde{f}$. One of the properties of the pigeonholed version is that for all $\g$, either $\|\tilde{f}_\g\|_\infty\sim A$ or $\|\tilde{f}_\g\|_\infty=0$, for some constant $A$. 

Apply Theorem \ref{alphamain} to the function $\tilde{f}/A$ to obtain the inequality
\[ (\a/A)^p|U_\a|\lesssim_\e R^{20\e}R^{2\b+1} \sum_\g\|\tilde{f}_\g/A\|_{L^2(\R^3)}^2.  \]
It remains to note that by \eqref{Lpprop} from the pigeonholing proposition, 
\[ A^{p-2}\|\tilde{f}_\g\|_{L^2(R^{\max(2\b,1)})}^2\lesssim R^{6\e} A^p\#\tilde{\T}_\g R^{\b+2\b+1}\lesssim R^{6\e}\|\tilde{f}_\g\|_{L^p(\R^3)}^p. \]
Since $|\tilde{f}_\g|\lesssim |f_\g|$ for each $\g$, this concludes the proof. 

\end{proof}

\begin{proposition}\label{genp} Proposition \ref{critexp} implies Theorem \ref{main} . 
\end{proposition}

\begin{proof}[Proposition \ref{critexp} implies Theorem \ref{main}] Let $p\ge 2$. Repeat the initial steps in the proof of Proposition \ref{critexp} so that it suffices to prove 
\[ \a^p|U_\a|\lesssim_\e R^\e (R^{\b(\frac{p}{2}-1)}+R^{\b(p-4)-1})\sum_\g\|f_\g\|_{L^p(\R^3)}^p\]
where $f$ has been pigeonholed and localized as in Proposition \ref{wpd}. First suppose that $2\le p\le 6+\frac{2}{\b}$. By Proposition \ref{critexp}, we have
\[ \a^{6+\frac{2}{\b}}|U_\a|\lesssim_\e R^\e R^{2\b+1}\sum_\g\|f_\g\|_{L^{6+\frac{2}{\b}}(\R^3)}^{6+\frac{2}{\b}}  . \]
Write $A\sim\max_\g\|f_\g\|_\infty$. We would be done if $R^{2\b+1}A^{6+\frac{2}{\b}-p}\lesssim R^{\b(\frac{p}{2}-1)}\a^{6+\frac{2}{\b}-p}$, which simplifies to $R^{\frac{\b}{2}}A\lesssim \a$. If this does not hold, then using $L^2$ orthogonality, 
\[\a^p|U_\a|\lesssim R^{\b(\frac{p}{2}-1)}A^{p-2}\sum_\g\|f_\g\|_2^2.  \]
By \eqref{Lpprop}, $A^{p-2}\|f_{\g}\|_2^2\lesssim R^{3\e}\|f_\g\|_p^p$, which finishes this case.

Next, assume that $6+\frac{2}{\b}\le p$. Then by Proposition \ref{critexp}, 
\[ \a^p|U_\a|\lesssim_\e R^\e  R^{2\b+1}\sum_\g\a^{p-6-\frac{2}{\b}}\|f_\g\|_{6+\frac{2}{\b}}^{6+\frac{2}{\b}}. \]
We would be done if $ R^{2\b+1}\a^{p-6-\frac{2}{\b}}\lesssim R^{\b(p-4)-1}A^{p-6-\frac{2}{\b}}$, which simplifies to $\a\lesssim R^\b A$. Since $\a\lesssim |f(x)|=|\sum_\g f_\g(x)|$ and $\#\g\lesssim R^\b$, this is true. 

\end{proof}

\bibliographystyle{alpha}
\bibliography{AnalysisBibliography}

\end{document}